\documentclass[11pt]{amsart}
\usepackage[leqno]{amsmath}
\usepackage{amsfonts, amssymb,srcltx, amsopn, color}

\usepackage[alphabetic]{amsrefs}
\usepackage{hyperref}
\usepackage{MnSymbol}

\usepackage{graphicx}

\usepackage{epsfig}
\usepackage{color}
\usepackage{amsmath}
\usepackage{amssymb}
\usepackage{graphicx}
\usepackage{xspace}
\usepackage{enumerate}

\usepackage{epsf}
\usepackage{psfrag}

\newtheorem{lemma}{Lemma}[section]

\newtheorem{proposition}[lemma]{Proposition}
\newtheorem{theorem}{Theorem}
\newtheorem{corollary}{Corollary}

\newtheorem{claim}[lemma]{Claim}
\theoremstyle{definition}

\newtheorem{remark}[lemma]{Remark}

\usepackage{color}
\newcommand{\commentout}[1]{}

\newcommand{\ta}{\ensuremath{\widetilde{a}\xspace}}
\newcommand{\tb}{\ensuremath{\widetilde{b}\xspace}}
\newcommand{\tc}{\ensuremath{\widetilde{c}\xspace}}
\newcommand{\td}{\ensuremath{\widetilde{d}\xspace}}

\newcommand{\tu}{\ensuremath{\widetilde{u}\xspace}}
\newcommand{\tv}{\ensuremath{\widetilde{v}\xspace}}
\newcommand{\tw}{\ensuremath{\widetilde{w}\xspace}}
\newcommand{\tx}{\ensuremath{\widetilde{x}\xspace}}
\newcommand{\ts}{\ensuremath{\widetilde{s}\xspace}}
\newcommand{\ty}{\ensuremath{\widetilde{y}\xspace}}
\newcommand{\tz}{\ensuremath{\widetilde{z}\xspace}}
\newcommand{\tB}{\ensuremath{\widetilde{B}\xspace}}
\newcommand{\tS}{\ensuremath{\widetilde{S}\xspace}}
\newcommand{\tG}{\ensuremath{\widetilde{G}\xspace}}
\newcommand{\tX}{\ensuremath{\widetilde{X}\xspace}}

\newcommand{\bX}{\ensuremath{{X}\xspace}}

\newcommand{\wt}{\widetilde}

\begin{document}

\thispagestyle{empty}
\centerline{\Large\bf On two conjectures of Maurer concerning basis graphs of matroids}

\vspace{10mm}

\centerline{{\sc J\'er\'emie Chalopin$^{\small 1}$, Victor
Chepoi$^{\small 1}$,} and {\sc Damian
Osajda$^{\small {2,3}}$} }

\vspace{3mm}

\medskip
\begin{small}

\medskip
\centerline{$^{1}$Laboratoire d'Informatique Fondamentale, Aix-Marseille Universit\'e and CNRS,}
\centerline{Facult\'e des Sciences de Luminy, F-13288 Marseille Cedex 9, France}

\centerline{\texttt{\{jeremie.chalopin, victor.chepoi\}@lif.univ-mrs.fr}}

\medskip
\centerline{$^{2}$Universit\"at Wien, Fakult\"at f\"ur Mathematik}
\centerline{Oskar-Morgenstern-Platz 1, 1090 Wien, Austria}

\medskip
\centerline{$^{3}$Instytut Matematyczny, Uniwersytet Wroc{\l}awski,}
\centerline{pl. Grunwaldzki 2/4, 50-384 Wroc{\l}aw, Poland}

\centerline{\texttt{dosaj@math.uni.wroc.pl}}

\end{small}

\bigskip\bigskip\noindent {\footnotesize {\bf Abstract.}  We characterize $2$--dimensional complexes associated canonically with basis graphs of matroids as simply connected triangle-square complexes satisfying some local conditions. This proves a version of a (disproved) conjecture by Stephen Maurer (Conjecture 3 of S. Maurer, Matroid basis graphs I, JCTB 14 (1973), 216--240).  We also  establish Conjecture 1 from the same paper about the redundancy of the conditions in the characterization of basis graphs. We indicate positive-curvature-like aspects of the local properties of the studied complexes. We characterize similarly the corresponding $2$--dimensional complexes of even $\Delta$--matroids.}

\thispagestyle{empty}

\section{Introduction}
\label{s:intro}

Matroids constitute an important unifying structure in
combinatorics, algorithmics, and combinatorial optimization ---
cf.\ e.g.\ \cite{Oxley} and references therein.
A \emph{matroid} on a finite set of elements $I$ is a collection
$\mathcal B$ of subsets of $I,$ called \emph{bases,}  which satisfy the
following exchange property: for all $A,B\in {\mathcal B}$ and $a\in A\setminus B$
there exists $b\in B\setminus A$ such that $A\setminus
\{ a\}\cup \{ b\}\in {\mathcal B}$ (the base $A\setminus \{ a\} \cup \{ b\}$ is obtained
from the base $A$ by an \emph{elementary exchange}). The \emph{basis graph} $G=G({\mathcal B})$
of a matroid $\mathcal B$ is the
graph whose vertices are the bases of $\mathcal B$ and edges are the
pairs  $A,B$ of bases differing by an elementary exchange, i.e., $\vert
A\Delta B\vert=2.$   Basis graphs faithfully represent their
matroids \cite{HoNoTo,Mau}, thus studying the basis graph amounts
to studying the matroid itself.

By the exchange property, basis graphs  are connected. For any two bases $A$ and $B$ at distance $2$
there exist at most four bases adjacent to $A$ and $B$: if $A\setminus B=\{ a_1,a_2\}$
and $B\setminus A=\{ b_1,b_2\}$, then these bases have the form $A\setminus \{ a_i\}\cup \{ b_j\}=B\setminus \{ b_j\}\cup \{ a_i\}$
for $i,j\in \{ 1,2\}.$ On the other hand, the exchange property ensures that at least one of the pairs
$A\setminus \{ a_1\}\cup \{ b_1\}, A\setminus \{ a_2\}\cup \{ b_2\}$ or
$A\setminus \{ a_1\}\cup \{ b_2\}, A\setminus \{ a_2\}\cup \{ b_1\}$ must be bases.
Together with $A$ and $B$, this pair of bases $C,C'$ induce a square in the basis graph. Therefore,
$A$ and $B$, together with their common neighbors induce a square, a pyramid, or an octahedron, i.e., basis graphs satisfy what we call
the \emph{interval condition}. The exchange property of bases also shows that if $A,C,B,C'$ induce a square in the basis graph,
then for any other base $D\in \mathcal B,$
the equality $d(D,A)+d(D,B)=d(D,C)+d(D,C')$ holds (i.e., the total  number of elementary exchanges to transform $D$ to $A$ and $B$
equals to the total number of exchanges to
transform $D$ to $C$ and $C'$). Following \cite{Mau}, we call this property of basis graphs the \emph{positioning condition}.
Finally, by Lemma 1.8 of \cite{Mau}, the subgraph
induced by all bases adjacent to a given base is the line graph of a finite bipartite graph; we will call it
the \emph{link condition}.

In \cite[Theorem 2.1]{Mau} Maurer characterized the basis graphs of matroids as connected graphs satisfying the three conditions
above --- see Theorem~\ref{Th_Mau1} in Section~\ref{s:Maurer} below for the
precise statement and for a stronger version of this characterization provided in \cite[Theorem 3.1]{Mau}.
Furthermore, in \cite[Theorem 3.5]{Mau}, it is established that under some additional conditions the link condition is
redundant and Conjecture 1 of \cite{Mau} asks if this is the case in
general. Our first result provides a positive answer to this conjecture. (Note that for a finite graph $G$ the
finiteness assumption on a link is trivially satisfied.)

\begin{theorem} \label{link} The link condition  is redundant for all basis graphs, in the following sense. A graph $G$ is the basis
graph of a matroid if and only if $G$ is connected, satisfies the interval and the positioning conditions, and has at least one
vertex with finitely many neighbors.
\end{theorem}

According to \cite{Bjo} (and implicitly introduced on pp.\ 237--239 in \cite{Mau}), the \emph{basis complex} $X=X({\mathcal B})$
of a matroid $\mathcal B$ is the $2$--dimensional cell complex whose $1$--skeleton is the basis graph $G$, and whose $2$--cells are the
triangles and the  squares of the basis graph.  We call this complex also the \emph{triangle-square complex} of $G$, and denote it by $X(G)$.

From the characterization of basis graphs, Maurer deduced in \cite[Theorem 5.1]{Mau} that all basis  complexes of matroids are simply connected.
Consequently, he proposed (a natural from the topological viewpoint) Conjecture 3 of \cite{Mau}, stating that in the characterization of basis graphs the global (metric)
positioning condition
on $G$ can be replaced by the topological condition of simply connectedness  of the triangle-square complex $X(G)$ of $G$. Donald, Holzmann, and Tobey \cite{DHT}
(as well as Maurer in the personal communication to the
authors of \cite{DHT}) presented counterexamples to this conjecture (as well as to Conjecture 2 of \cite{Mau} about the eventual redundancy of the positioning condition),
i.e., simply connected
triangle-square complexes, satisfying the interval and the link conditions, but not being basis complexes --- cf.\ Section~\ref{s:rem} below.
Nevertheless, the main result of our paper is in the vein of Maurer's Conjecture 3 --- saying that triangle-square complexes of basis graphs of matroids may be characterized
as simply connected complexes satisfying some local conditions.

\begin{theorem}\label{main_th}
  For a graph $G$ the following conditions are equivalent:
  \begin{enumerate}[(i)]
    \item
    $G$ is the basis graph of a matroid;
    \item
    the triangle-square complex $X(G)$ is simply connected and every ball of radius $3$ in $G$ is isomorphic to a ball of radius $3$ in
    the basis graph of a matroid;
    \item
    the triangle-square complex $X(G)$ is simply connected, $G$
    satisfies the interval and the local positioning conditions, and
    $G$ contains at least one vertex with finitely many neighbors.
  \end{enumerate}
\end{theorem}

A formal definition of the local positioning condition is given in
the next section.  This condition, as well as the interval condition, are local because they concern at most quintets of vertices at distance $\le 3$ from each other.

Simple connectivity of basis  complexes of matroids  was used several times in the theory of ordinary and oriented matroids, in particular,
in the proof of Las Vergnas's theorem \cite{LasVer,BjLVStWhZi}
about the characterization of basis orientations of  ordinary matroids. This result was generalized in \cite{BjKoLo} to basis complexes of $3$--connected interval greedoids
and in \cite{Wen} to even $\Delta$--matroids. From  this result also follows that the
$2$--dimensional faces of the basis matroid polyhedron are equilateral
triangles or squares, i.e., the $2$--skeleton of the basis matroid polyhedron is a simply connected subcomplex of the basis complex, namely,
it comprises all triangles and a part of squares of this complex; cf.\ also \cite{BoGeWh}
(a \emph{basis matroid polyhedron} \cite{GeGoMcPhSe} is the convex hull of the characteristic vectors of bases of a matroid).
Moreover, Gelfand et al.\ \cite{GeGoMcPhSe} showed that the $1$--skeleton of a basis matroid polyhedron coincides with the basis graph
of the matroid.

Characterizing spaces by requiring they are simply connected and satisfy some local conditions is natural and appears often in the setting of a (very general) nonpositive curvature. In particular, in a simple but fundamental result, Gromov \cite{Gr} characterized the CAT(0) cubical complexes (i.e., cubical complexes with global nonpositive curvature) as simply connected  cubical complexes in which the links of vertices are flag. Many similar characterizations concerning widely understood nonpositive curvature appeared --- cf.\ e.g.\ \cite{BCCGO} for an example and for further references. Such characterizations are very useful, since they allow to construct objects out of just local conditions: Having a space satisfying given local conditions, its universal cover (whose existence and uniqueness follows from basic algebraic topology) is a simply connected space satisfying the same collection of local conditions. Note (compare also Corollaries~\ref{c:finuni}\&\ref{c:finfun} below) that constructing a complex satisfying our local conditions will finish after finitely many steps. Then either this complex or its finitely sheeted (universal) cover is the basis complex of a matroid.
As a matter of fact, building the universal cover of a triangle-square complex with a prescribed local behavior is our way to prove Theorem \ref{main_th} --- see Theorem~\ref{t:m2} in Section~\ref{s:main}.
Note however that our setting is opposite to the case of nonpositive curvature. Since basis graphs of matroids are finite (unlike universal covers of homotopically nontrivial spaces with nonpositive curvature), Theorem~\ref{main_th} implies immediately the following. (Note that the conditions in the statements below are local.)

\begin{corollary}
\label{c:finuni}
Let $G$ be a connected graph satisfying the interval and the local
positioning conditions, and having at least one vertex with finitely
many neighbors.  Then the $1$--skeleton of the universal cover
$\widetilde{X(G)}$ of its triangle-square complex $X(G)$ is the basis
graph of a matroid.  In particular, $\widetilde{X(G)}$ is a finite
complex.
\end{corollary}

\begin{corollary}
\label{c:finfun}
Let $G$ be a connected graph satisfying the interval and the local
positioning conditions, and having at least one vertex with finitely
many neighbors.  Then the fundamental group $\pi_1(X(G))$ of its
triangle-square complex $X(G)$ is finite.
\end{corollary}

Thus the collection of our local conditions may be treated as a kind of a positive curvature.
Our characterization might be seen as an analogue of e.g.\ the classical result of Myers \cite{Mye} characterizing spheres by
means of positive curvature. 

Our construction can be used to obtain a similar characterization of
basis graphs of even $\Delta$--matroids, for which an analogue of
Maurer's characterization is provided in \cite{Che_bas} --- see Theorem~\ref{th_Delta} in Section~\ref{s:rem}. This construction may also be useful to obtain
similar characterizations in other cases.
\medskip

\noindent
{\bf Article's structure.}
In the next section, we define the local conditions employed in the formulation of Theorem \ref{main_th} and we prove several auxiliary results. We also provide a
slight enhancement of the original Maurer's characterization. Theorem~\ref{link} is proved in Section~\ref{s:link}. Section~\ref{s:main} is devoted to the proof
of Theorem \ref{main_th}. We conclude in Section~\ref{s:rem} with some examples, in particular we analyze examples of non-basis graphs described in \cite{DHT}, and we extend Theorem \ref{main_th} to even $\Delta$--matroids.

\section{Preliminaries}

\subsection{Graphs.}  All graphs $G=(V,E)$ occurring in this paper are undirected, connected,
without loops or multiple edges, and not necessarily finite (unless stated otherwise).
The \emph{distance} $d(u,v)$ between
two vertices $u$ and $v$ is the length of a shortest $(u,v)$--path, and
the \emph{interval} $I(u,v)$ between $u$ and $v$ consists of all
vertices on shortest $(u,v)$--paths, that is, of all vertices
(metrically) \emph{between} $u$ and $v$: $I(u,v)=\{ x\in V:
d(u,x)+d(x,v)=d(u,v)\}.$ If $d(u,v)=2,$ then we will call $I(u,v)$ a $2$--\emph{interval}.
For two vertices $u$ and $v$ of a graph $G$,
we will write $u\sim v$ if $u$ and $v$ are adjacent and $u\nsim v$, otherwise.
Having vertices $u,v_1,v_2,\ldots,v_k$, we will write $u\sim v_1,v_2,\ldots,v_k$ (respectively,
$u\nsim v_1,v_2,\ldots,v_k$) if $u\sim v_i$ (respectively, $u\nsim v_i$), for every $i$.
For a vertex $v$ of
a graph $G$ and an integer $r\ge 1$, we will denote  by $B_r(v,G)$ the \emph{ball} in $G$ (and the subgraph induced
by this ball)  of radius $r$ centered at  $v$, i.e., $B_r(v,G)=\{ x\in V: d(v,x)\le r\}.$
As usual, $N(v)=B_1(v,G)\setminus\{ v\}$ denotes the set of neighbors of a vertex $v$ in $G$.
The \emph{link} of $v \in V(G)$ is the subgraph of $G$ induced by $N(v)$.

A {\em wheel} $W_k$ is a graph obtained by connecting a single
vertex --- the {\em central vertex} --- to all vertices of the $k$--cycle;
the \emph{almost wheel} $W_4^{-}$ is the graph obtained from $W_4$ by
deleting a spoke (i.e., an edge between the central vertex and a
vertex of the $4$--cycle).
A \emph{pyramid} is the $4$--wheel $W_4$.
A \emph{triangle} and a \emph{square} of $G$ are
subgraphs of $G$ which are induced $3$-- and $4$--cycles. An \emph{octahedron} is the $1$--skeleton of the $3$--dimensional octahedron,
i.e., it is the complete  graph $K_6$ minus a perfect matching. The following two graphs were shown in \cite{Mau}
to be forbidden (as induced subgraphs) in basis graphs of matroids. A \emph{propeller} with \emph{shaft} $uv$ and \emph{tips} $x,y,z$ is the
graph $P$ defined by $V(P) =\{u,v,x,y,z\}$ and $E(P) =
\{uv,ux,uy,uz,vx,vy,vz\}$ (see Figure~\ref{fig-propeller-halfbook}, left). A \emph{half open book}
is the graph $B$ defined by $V(B)=\{u,v,w,x,y,z\}$ and $E(B) = \{uv,ux,vw,vz,wy,xy,xz,yz\}$ (see
Figure~\ref{fig-propeller-halfbook}, right).

We continue with definitions of local and global conditions used in Maurer's and our characterizations of basis graphs.
A graph $G$ satisfies the \emph{interval condition} if each $2$--interval induces a square, a pyramid, or an octahedron.
A graph $G$ satisfies the \emph{link condition at vertex $v$}, denoted LC$(v)$ if the link of $v$ in $G$ is the line graph of a finite bipartite
graph. A graph $G$ satisfies the \emph{link condition} if $G$ satisfies LC$(v)$ for all vertices $v$. Next, we introduce three
global metric conditions  with respect to a given basepoint $v$:

\emph{Triangle condition} TC($v$):  for any two adjacent vertices $u,w$ of $G$ with
$d(v,u)=d(v,w)=k\geq 2$ there exists $x \sim u,w$ such that $d(v,x)=k-1$.

\emph{Square-pyramid condition} SPC$(v)$: for any three vertices $u,w,w'$ of $G$ with
$u\sim w,w'$ and  $2=d(w,w')\le d(v,u)=d(v,w')+1=d(v,w)+1=k+1,$ either there
exists $x\sim w, w'$ such that
$d(v,x)=k-1$, or there exists $x \sim u, w, w'$ and $x' \sim u, w,w'$ such that $x\nsim x'$, and $d(x,v)=d(x',v)=k$.

\emph{Positioning condition} PC$(v)$: for each square $u_1u_2u_3u_4$ of $G$, we have $d(v,u_1)+d(v,u_3)=d(v,u_2)+d(v,u_4)$.

A graph $G$ satisfies the \emph{triangle}, the \emph{square-pyramid}, or the \emph{positioning conditions} if  $G$ satisfies TC$(v),$ SPC$(v),$ or PC$(v),$ respectively,
for all  vertices $v$. A graph $G$ satisfies the \emph{local triangle condition} if for every $v,u,w$ with $u\sim w$ and
$d(v,u)=d(v,w)= 2$ there exists $x \sim v,u,w$. A graph $G$ satisfies the \emph{local positioning condition} if for each square $u_1u_2u_3u_4$ and each vertex $v$ such that $d(v,u_1)=d(v,u_3)=2$, we have $d(v,u_2)+d(v,u_4)=4$.

\begin{figure}[t]
\begin{center}
\scalebox{0.7}{\includegraphics{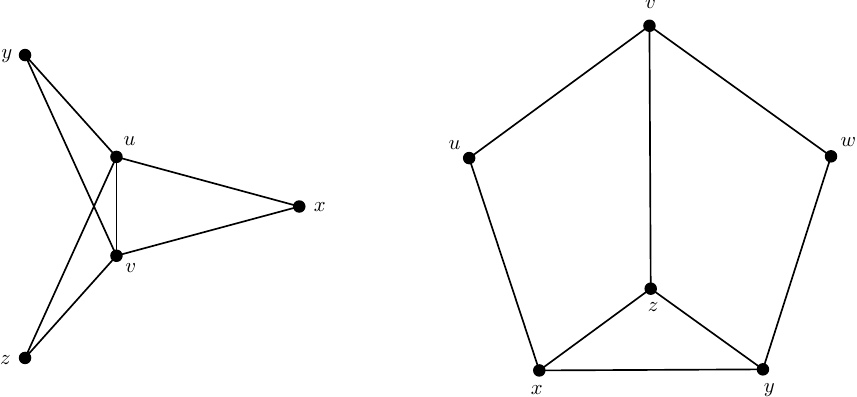}}%
\end{center}
\caption{A propeller (on the left) and a half open book (on the right).}
\label{fig-propeller-halfbook}
\end{figure}

\begin{lemma}
\label{l:noprop}
  If $G$ satisfies the interval and the local positioning conditions,
  then $G$ satisfies the local triangle condition, and $G$ does not
  contain propellers and half open books as induced subgraphs.
\end{lemma}

\begin{proof}
Consider three vertices $u,v,w$ such that $u\sim w$ and $d(u,v) =
d(w,v) =2$. By the interval condition, there exist $x,x' \sim u,v$
such that $x \nsim x'$. If $w \nsim x,x'$, then $d(w,x) = d(w,x') =
2$ and $d(w,u)+d(w,v) = 3$, contradicting the local positioning
condition. Consequently, either $x \sim u,v,w$ or $x' \sim u,v,w$ and
thus $G$ satisfies the local triangle condition.

Consider three triangles $uvx$, $uvy$, $uvz$, all three sharing the
common edge $uv$. Suppose that $x \nsim y$ (see
Figure~\ref{fig-propeller-halfbook}, left). By the interval condition,
there exists $t \sim x,y$ such that $t \sim u, t\nsim v$ or $t \sim
v, t\nsim u$, say $t \sim u, t\nsim v$. Since $d(z,t) \leq 2$,
$d(z,v)+d(z,t) \leq 3$.  By the local positioning condition applied to
the square $vxty$ and the basepoint $z$, we get that $d(x,z)+d(z,y)
\leq 3$, and consequently, $z$ is adjacent to $x$ or $y$. Thus, $G$
does not contain propellers.

Suppose now that $G$ contains a half open book where $xyz$ is a
triangle and $uxzv$ and $vzyw$ are squares (see Figure~\ref{fig-propeller-halfbook}, right). Then, considering the
square $vwyz$, we have $d(u,w)=d(u,z)=2$ and $d(u,v)+d(u,y) = 3$,
contradicting the local positioning condition with respect to $u$.
\end{proof}

\subsection{Triangle-square complexes.}  In this paper, we consider only triangle-square complexes,
a particular class of $2$--dimensional cell complexes.  Although most of the notions presented below can be
defined for all cell complexes and some of them for topological spaces, we will introduce them
only for triangle-square complexes.

A \emph{triangle-square complex} is a $2$--dimensional cell complex $X$
in which all $2$--cells are triangles or squares. For a
triangle-square complex $X$, denote by $V(X)$ and $E(X)$ the set of
all $0$--dimensional and $1$--dimensional cells of $X$ and call the
pair $G(X)=(V(X),E(X))$ the \emph{$1$--skeleton} of $X$, or the
\emph{underlying graph}. Conversely, for a graph $G$ one can derive a
triangle-square complex $X(G)$ by taking all vertices of $G$ as
$0$--cells, all edges of $G$ as $1$--cells, and all triangles and
squares of $G$ as $2$--cells of $X(G)$. Then $G$ is the $1$--skeleton
of $X(G)$. A triangle-square complex $X$ is a \emph{flag complex} if
the triangular and the square cells of $X$ are exactly the triangles
and the squares of its $1$--skeleton $G(X)$; a triangle-square flag
complex $X$ can therefore be recovered from its underlying graph
$G(X)$. The \emph{star} $\mbox{St}(v,X)$ of a vertex $v$ in a
triangle-square complex $\bX$ is the subcomplex consisting of the
union of all cells in $X$ containing $v$.

As morphisms between triangle-square complexes we consider all \emph{cellular maps}, i.e.,
maps sending (linearly) cells to cells. An \emph{isomorphism} is a bijective
cellular map being a linear isomorphism (isometry) on each cell. A
\emph{covering (map)} of a cell complex $X$ is a cellular surjection $p\colon
\widetilde{X} \to X$ such that $p|_{\mbox{St}(\tv,\widetilde{X})}$ is
an isomorphism onto its image for every vertex $\tv$ in $\tX$; compare \cite[Section 1.3]{Hat}.
The space $\widetilde{X}$ is then called a \emph{covering space}.
A \emph{universal cover} of $X$ is a simply connected covering space $\widetilde{X}$. It
is unique up to an isomorphism; cf.\ \cite[page 67]{Hat}. In particular, if $X$ is simply connected, then
its universal cover is $X$ itself. (Note that $X$ is connected iff $G(X)$ is connected, and $X$ is
\emph{simply connected} if every continuous map $S^1\to X$ is null-homotopic).

The following lemma, that is important in the proof of Theorem \ref{main_th} presented in
Section~\ref{s:main}, also provides an alternative proof of
Maurer's Theorem 5.1 from \cite{Mau}, establishing simple connectedness of basis complexes.

\begin{lemma} \label{l:simconn} Let ${X}$ be a
  triangle-square flag complex such that $G({X})$ satisfies the
  triangle and the square-pyramid conditions TC$(v)$ and SPC$(v)$, for some
  basepoint $v$. Then ${X}$ is simply connected.
\end{lemma}

\begin{proof}
A \emph{loop} in ${X}$ is a sequence $(w_1,w_2,\ldots,w_k,w_1)$ of
vertices of ${X}$ consecutively joined by edges in $G(X)$. To prove the
lemma it is enough to show that every loop in ${X}$ can be freely
homotoped to a constant loop $v$.  By contradiction, let $A$ be the
set of loops in $G({X})$, which are not freely homotopic to $v$,
and assume that $A$ is non-empty. For a loop $\alpha\in A$ let $r(\alpha)$
denote the maximal distance $d(w,v)$ of a vertex $w$ of $\alpha$ from the
basepoint $v$.  Clearly $r(\alpha)\geq 2$ for any loop $\alpha\in A$ (otherwise
$\alpha$ would be null-homotopic). Let $B\subseteq A$ be the set of loops
$\alpha$ with minimal $r(\alpha)$ among loops in $A$. Let $r:= r(\alpha)$ for some
$\alpha\in B$. Let $C\subseteq B$ be the set of loops having minimal number
$e$ of edges in the $r$--sphere around $v$, i.e., with both endpoints at
distance $r$ from $v$. Further, let $D\subseteq C$ be the set of loops
with the minimal number $m$ of vertices at distance $r$ from $v$.

Consider a loop $\alpha=(w_1,w_2,\ldots,w_k,w_1)\in D$.
We  can assume without loss of generality that $d(w_2,v)=r$. We treat separately the three following cases.
\medskip

\noindent \emph{Case 1: $d(w_1,v)=r$ or $d(w_3,v)=r$.}  Assume without
loss of generality that $d(w_1,v)=r$. Then, by the triangle condition
$TC(v)$, there exists a vertex $w\sim w_1,w_2$ with $d(w,v)=r-1$.
Observe that the loop $\alpha'=(w_1,w,w_2,\ldots,w_k,w_1)$ belongs to $B$ --- in ${X}$ it is freely homotopic to $\alpha$ by a homotopy going through
the triangle $ww_1w_2$. The number of edges of $\alpha'$ lying on the
$r$--sphere around $v$ is less than $e$ (we removed the edge
$w_1w_2$). This contradicts the choice of the number $e$.
\medskip

\noindent \emph{Case 2: $d(w_1,v)=d(w_3,v)=r-1$ and $w_1\sim w_3$.} Then the loop $\alpha'=(w_1,w_3,\ldots,w_k,w_1)$ is homotopic to $\alpha$ via the triangle $w_1w_2w_3$. Thus $\alpha'$ belongs to $C$ and the number of its vertices at distance $r$ from $v$ is $m-1$. This contradicts the choice of the number $m$.

\medskip

\noindent \emph{Case 3: $d(w_1,v)=d(w_3,v)=r-1$ and $d(w_1,w_3)=2$.}
By the square-pyramid condition $SPC(v)$, there exists a vertex $w\sim w_1,w_3$ with $d(w,v) \leq r-1$.
Again, the loop $\alpha'=(w_1,w,w_3,\ldots,w_k,w_1)$ is freely homotopic to $\alpha$ (via the square $w_1w_2w_3w$, or the triangles $ww_1w_2$ and $ww_2w_3$). Thus $\alpha'$
belongs to $C$ and the number of its vertices at distance $r$ from $v$ is equal to $m-1$. This contradicts
the choice of the number $m$.
\medskip

In all cases above we get a contradiction. It follows that the set $A$ is empty and hence the lemma is proved.
\end{proof}

\subsection{A note on Maurer's characterizations}
\label{s:Maurer}

Now we formulate the main characterizations of basis graphs presented in Theorems 2.1 and 3.1 of \cite{Mau}. Both these results were proved in \cite{Mau} for finite graphs. However, the analysis of
the proof shows that one does not need to assume that the graphs are finite. Indeed, the result shows that if a graph $G$ satisfies Maurer's conditions, then $G$ is necessarily finite.

\begin{theorem} \cite[Theorems 2.1\&3.1]{Mau} \label{Th_Mau1} For a graph $G$ the following statements are equivalent:
\begin{enumerate}[(i)]
\item $G$ is the basis graph of  a matroid;
\item $G$ is connected, satisfies the interval and the positioning conditions, and some vertex  of $G$ satisfies the link condition (in particular, has finite degree);
\item $G$ is connected, satisfies the interval condition, does not contain propellers and half open books, and for some vertex $v$, the graph $G$ satisfies the link condition \mbox{LC}$(v)$ (in particular, $v$ has finite degree) and the positioning condition \mbox{PC}$(v)$.
\end{enumerate}
\end{theorem}

\begin{proof}
The implications $(i) \Rightarrow (ii) \Rightarrow (iii)$ are clear and are showed in \cite{Mau}.
For the implication $(iii) \Rightarrow (i)$, the main part of the proof of Theorems 2.1 and 3.1 of \cite{Mau} is to encode the vertices $x\in V$ of $G$ with pairwise distinct, equicardinal
sets (labels) $S_x$ such that two vertices $x$ and $y$ are adjacent in $G$ if and only if $|S_x\Delta S_y|=2$. The only place where the finiteness assumption
is used is the beginning of the encoding,  namely, to find an encoding for the basepoint $v$ and its neighbors. By the link condition, the link of $v$ is the line graph of a finite bipartite graph  $H=(B_0\cup B,E)$. Then, set $S_v:=B_0$ and for each vertex $x\in N(v),$ if $x$ corresponds to the
edge $b_0b$ of $H$ with $b_0\in B_0$ and $b\in B,$ then set
$S_x:=B_0 \setminus\{ b_0\}\cup \{ b\}$. The encoding is then propagated level-by-level to the whole graph $G$ using the interval condition, the
positioning condition PC$(v)$, and the fact that $G$ does not contain propellers and half open books. Each vertex $x$ of $G$ is encoded with a subset $S_x$ of $B_0\cup B$ of size $|B_0|$.
Since there exists only a finite number of such subsets, we conclude that $G$ is finite and is the basis graph of a matroid.
\end{proof}

\section{Proof of Theorem \ref{link}}
\label{s:link}
In this section we prove Theorem \ref{link}, which establishes Conjecture 1 of Maurer \cite{Mau}.  The proof is
a direct consequence of  Theorem~\ref{Th_Mau1} above and the following result.

\begin{theorem}\label{link1}
Let $G$ be a graph, and let $v$ be a vertex adjacent to finitely many vertices in $G$.
If $G$ satisfies the interval condition and $G$ does not contain
propellers and half open books, then $G$ satisfies the link condition LC$(v)$ at vertex $v$.
\end{theorem}

\begin{proof}
Let  $G'$ be the link of $v$ in $G$. It is well known that $G'$ is the line graph of a bipartite
graph if and only if $G'$ does not contain induced claws $K_{1,3}$, diamonds
$K_4 - e$, and odd induced cycles $C_{2k+1}, k\ge 2$.

If $G'$ contains a claw, then this $K_{1,3}$ together with $v$ induces in $G$
a propeller. Analogously, if $G'$ contains a diamond
with vertices $a,b,c,d$ such that $c\nsim d,$ then in $G$ the interval $I(c,d)$
contains a triangle $abv$, which is impossible by the
interval condition. If $G'$ contains an induced odd cycle, then this cycle together with $v$ induces in
$G$ an odd wheel $W_{2k+1}$ with $k \geq 2$, which is impossible by the next Proposition~\ref{prop-no-odd-wheel}.
\end{proof}

\begin{proposition}\label{prop-no-odd-wheel}
If $G$ satisfies the interval condition and  $G$ does not contain
propellers and half open books, then $G$ does not contain any odd wheel
$W_{2k+1},$ $k \geq 2$.
\end{proposition}

In the rest of this section, we will prove
Proposition~\ref{prop-no-odd-wheel}.  Consider the smallest $k \geq 2$
such that $G$ contains an induced odd wheel $W_{2k+1}$. Let $c$ be the center of
the wheel, and let $v_0,\ldots, v_{2k}$ be the vertices of the cycle
of the wheel such that for every $i$, $v_i \sim v_{i+1}$ (here and in the rest of this section all additions are performed modulo $2k+1$).

\begin{lemma}\label{lem-defx}
For every $i,j$ such that $v_i \nsim v_j$, there exists a unique
$x_{i,j} \sim v_i,v_j$ such that $x_{i,j} \notin \{c,v_0,
\ldots,v_{2k}\}$. Moreover, the following properties are satisfied:
\begin{enumerate}
\item $x_{i,j} \nsim c$;
\item $x_{i,j} \sim v_k$ with $v_k \notin \{v_i, v_j\}$ if and only if $v_k
  \sim v_i, v_j$.
\end{enumerate}
\end{lemma}

\begin{proof}
By symmetry, we can assume that $i = 1$ and  $3 \leq j \leq k+1$.
If $j = 3$, by the interval condition there exists $x \notin \{c, v_2\}$
such that $x \sim v_1,v_3$.   If $j \geq 4$, by the interval condition
between $v_1$ and $v_j$, there exists $x\sim v_1,v_j$ with $x \neq c$. In both cases,
$x \notin \{c,v_0,\ldots,v_{2k}\}$. We first show that $x \nsim c$.

\begin{claim} \label{claim_wheel}
$x \nsim c$.
\end{claim}

\begin{proof}
Suppose that there exists $x \sim v_1,v_j,c$.  Let $x
\sim v_m$ for some $m \notin \{1,j\}$. Consider the three triangles $cxv_1,
cxv_j$, and $cxv_m$, all three sharing the common edge $cx$. Since $G$
does not contain propellers, $v_m$ is a neighbor of $v_1$ or
$v_j$. Note that if $v_m \sim v_1,v_j$, then $m=2$ and $j=3$, but
then the interval $I(v_1,v_3)$ contains a triangle
$cxv_2$. Consequently, either $v_m \sim v_1$ and $v_m \nsim v_j$, or
$v_m \sim v_j$ and $v_m \nsim v_1$.

Consider the  triangles $cv_1x$, $cv_1v_2$, and $cv_1v_0$. Since
$G$ has no propellers, either $x \sim v_2$, or $x \sim v_0$. By
the previous remark, $x$ cannot be adjacent to both $v_0$ and
$v_2$. Up to renaming the vertices, we can assume that $x \sim
  v_2$. For the same reasons, we can assume that $x \sim v_{j+1}$.

Consequently, $x \sim c, v_1, v_2, v_j, v_{j+1}$ and $x$ is not
adjacent to any other vertex of the wheel. Thus $c$ and the cycle
$v_2v_3\ldots v_jx$ form the wheel $W_j$, while $c$ and the cycle
$v_{j+1}\ldots v_{2k}v_0v_1x$ form the wheel
$W_{2k+3-j}$. Since $j$ or $2k+3-j$ is odd and strictly smaller
than $2k+1$, we get a contradiction with the choice of $k$,
except if $j=3$. In the latter case, the
interval $I(v_1,v_3)$ contains a triangle $cxv_2$, a
contradiction. This establishes Claim \ref{claim_wheel}.
\end{proof}

Hence, if $x \sim v_1, v_j$, then $x\nsim c$. Then the interval
condition for $v_1$ and $v_j$ ensures that $x$ is unique. Suppose now
that $x \sim v_m$ for some $m \notin \{1,j\}$. By the interval
condition between $c$ and $x$, and since $v_1 \nsim v_j$, we get that
$v_m \sim v_1,v_j$, i.e., $m=2$ and $j=3$ (since we assumed that  $3 \leq j
\leq k+1$). Conversely, assume that $v_m \sim v_1,v_j$, i.e., $m=2$
and $j=3$. Then $c,x,v_2$ belong to the interval $I(v_1,v_3)$ and, by the interval condition,
$x \sim v_2$ since $x \nsim
c$. This finishes the proof of the lemma.
\end{proof}

In the following, for any $v_i \nsim v_j$, let $x_{i,j}$ be the
unique vertex $x_{i,j} \sim v_i,v_j$ such that $x_{i,j} \notin
\{c,v_0, \ldots,v_{2k}\}$. From Lemma~\ref{lem-defx}(2), we know
that for every $i,j,i',j'$ such that $\{i,j\}\neq \{i',j'\}$, we have $x_{i,j}
\neq x_{i',j'}$.

\begin{lemma}\label{lem-x1ix1j}
For any $v_i, v_j, v_m$, we have $x_{i,j} \sim x_{i,m}$ if and only if $v_j \sim
v_m$.
\end{lemma}

\begin{proof}
Note that since $x_{i,j}, x_{i,m}$ are defined, $v_i \nsim
v_j,v_m$.
Consider the subgraph of $G$ induced by $c,v_i,v_j,v_m,x_{i,j},x_{i,m}$. Observe that
$cv_ix_{i,j}v_j$ and $cv_ix_{i,m}v_m$ are squares, since $v_i \nsim
v_j,v_m$, and from Lemma~\ref{lem-defx} we have $c \nsim x_{i,j}, x_{i,m}$,
$v_j \nsim x_{i,m}$, and $v_m \nsim x_{i,j}$. Since $G$ does not
contain half open books, $v_j \sim v_m$ if and only
if $x_{i,j} \sim x_{i,m}$.
\end{proof}

\begin{lemma}\label{lem-wi}
For any $v_i, v_j$, $x_{i-1,i+1} \sim x_{i,j}$ if and only if $j =
i-2$, or $j = i+2$.
\end{lemma}

\begin{proof}
By symmetry, we can assume that $i = 1$ and $2 \leq j \leq
k+1$. Note that since $x_{i,j}$ exists, $j \geq 3$. Recall that
$x_{0,2} \sim v_{0},v_1,v_{2}$. First assume that $j =3$; recall that $x_{1,3} \sim
v_1,v_{2},v_3$. Consider the  triangles $v_1v_2c$,
$v_1v_2x_{0,2}$, $v_1v_2x_{1,3}$, all three sharing the common edge
$v_1v_2$. Since $G$ does not contain propellers and $c \nsim
x_{0,2},x_{1,3}$, we get that $x_{0,2} \sim x_{1,3}$.

Suppose now that there exists an index $j$ such that $x_{1,j} \sim
x_{0,2}$. Consider the triangles $v_1x_{0,2}v_0$,
$v_1x_{0,2}v_2$, $v_1x_{0,2}x_{1,j}$, all three sharing the common edge
$v_1x_{0,2}$. Since $v_0 \nsim v_2$ and $G$ does not contain
propellers, either $x_{1,j} \sim v_0$ or $x_{1,j} \sim v_2$. Since $3
\leq j \leq k+1$, by Lemma~\ref{lem-defx}, the only possibility is
$j=3$.
\end{proof}

\begin{figure}[t]
\begin{center}
\scalebox{0.7}{\includegraphics{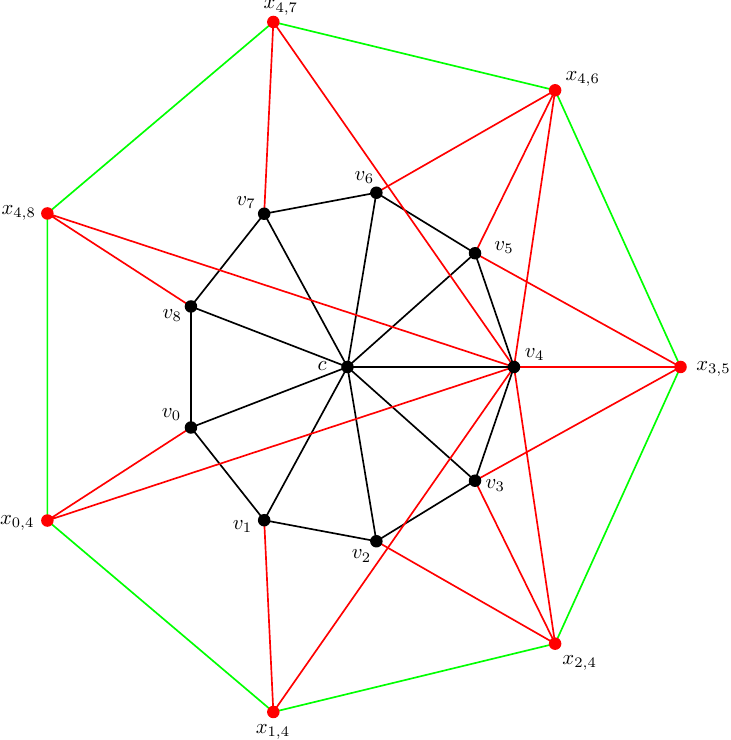}}%
\end{center}
\caption{The wheel $W_{7}$ around $v_4$ obtained from $W_9$ by
  Lemma~\ref{lem-contradiction}.}
\label{preuve-W2k+1}
\end{figure}

\begin{lemma}\label{lem-no-W5}
 $G$ does not contain any $W_5$, i.e., $k > 2$.
\end{lemma}

\begin{proof}
Suppose by way of contradiction  that $k=2$. By Lemma~\ref{lem-defx}, $v_2 \nsim x_{1,4}$.  Consider the interval
$I(v_2,x_{1,4})$. By Lemma~\ref{lem-defx}, $v_2 \sim x_{0,2}, x_{1,3}, x_{2,4}$.
Lemma~\ref{lem-x1ix1j} implies that $x_{1,4} \sim x_{2,4}, x_{1,3}$ and $x_{0,2}
\sim x_{2,4}$. By Lemma~\ref{lem-wi}, $x_{0,2} \sim x_{1,3},
x_{1,4}$ and $x_{1,3} \sim x_{2,4}$. Consequently, the pairwise adjacent vertices $x_{0,2}, x_{1,3}, x_{2,4}$ belong to the interval
$I(v_2,x_{1,4})$, contrary to the interval condition.
\end{proof}

\begin{lemma}\label{lem-contradiction}
The vertices $x_{0,k},x_{1,k}, \ldots,x_{k-2,k},x_{k-1,k+1},x_{k,k+2},
x_{k,k+3} \ldots, x_{k,2k-1}, x_{k,2k}$ form an induced cycle $C$ of
length $2k-1$ of $G$ such that $v_k$ is adjacent to all vertices of $C$
(see Figure~\ref{preuve-W2k+1}).
\end{lemma}

\begin{proof}
By Lemma~\ref{lem-x1ix1j}, $x_{i,k} \sim x_{i+1,k}$ for every $0 \leq i \leq k-3$ and
$k+2 \leq i \leq 2k$. By
Lemma~\ref{lem-wi},  $x_{k-1,k+1} \sim x_{k-2,k},
x_{k,k+2}$. Hence, $x_{0,k}x_{1,k}
\ldots x_{k-2,k}x_{k-1,k+1}x_{k,k+2}x_{k,k+3} \ldots x_{k,2k-1}x_{k,2k}$ is
a cycle $C$ of $G$.  By Lemma~\ref{lem-x1ix1j}, $x_{i,k}$ is adjacent to $x_{j,k}$ if and
only if $j \in \{i-1,i+1\}$; consequently, $C$ does not contain chords of
the form $x_{i,k}x_{j,k}$. Since, by Lemma~\ref{lem-wi}, we have $x_{k-1,k+1} \nsim x_{k,j}$ when $j
\notin \{k-2, k+2\}$, we conclude that $C$ is an induced cycle of $G$.
By the definition of $x_{i,k}$, we have $v_k \sim x_{i,k}$ for every $i$, and $v_k
\sim x_{k-1,k+1}$, by Lemma~\ref{lem-defx} $(2)$.
\end{proof}

By Lemma~\ref{lem-no-W5}, we know that $k \geq 3$, and by
Lemma~\ref{lem-contradiction}, we have constructed a wheel $W_{2k-1}$,
contrary to our choice of $k$. This finishes the proof of
Proposition~\ref{prop-no-odd-wheel}.

\section{Proof of Theorem \ref{main_th}}
\label{s:main}

In this section, we present the proof of Theorem \ref{main_th} --- the
main result of our paper. Theorem \ref{main_th} presents a topological
characterization of basis complexes of matroids and shows that a specialized
form of Conjecture 3 of \cite{Mau} is true.  Note that the
implications $(i)\Rightarrow (ii) \Rightarrow (iii)$ are clear and
follow from~\cite{Mau}. Thus in what follows we focus on proving the
implication $(iii) \Rightarrow (i)$.

Consider a graph $G$ that satisfies the interval and local positioning
conditions, that has a vertex with finitely many neighbors, and such
that its triangle-square complex $X(G)$ is simply connected. From the following result, $G$
satisfies the positioning condition. Consequently, from
Lemma~\ref{l:noprop}, Theorem~\ref{link1}, and Theorem~\ref{Th_Mau1}, the graph $G$ is the basis graph of a matroid.

\begin{theorem}
\label{t:m2} Let $G$ be a connected graph satisfying the interval and
the local positioning conditions. Then the $1$--skeleton of the
universal cover $\widetilde{X(G)}$ of the triangle-square complex
$X(G)$ of $G$ satisfies the interval and the positioning conditions.
\end{theorem}

The rest of the current section is devoted to the proof of the above
theorem.  In the following, we consider a connected graph $G$ that
satisfies the interval and the local positioning conditions. By
Lemma~\ref{l:noprop}, $G$ satisfies the local triangle condition and
$G$ does not contain propellers and half open books as induced
subgraphs. We construct inductively the universal cover,
simultaneously exhibiting its various properties.

\begin{remark}
 Our proof follows closely (including much of notations) the proof of an analogous result from \cite{BCCGO}. Note however that the overall setting is totally different --- positive versus nonpositive curvature (see the introduction for more background). Thus, consequences of the two constructions (of the universal cover) are very different --- finite versus infinite (see Corollaries~\ref{c:finuni} and \ref{c:finfun}). Moreover, as for technical details, the current proof is much more involved.
\end{remark}

\subsection{Structure of the construction.}
\label{s:struct}
In this subsection we describe our inductive construction of the universal cover and we set the basis for the induction.
\medskip

\noindent
We construct
the universal cover $\widetilde X:=\widetilde{{X(G)}}$ of $X:={X(G)}$ as an increasing
union $\bigcup_{i\ge 1} \widetilde{{X}}_i$ of triangle-square
complexes. The complexes $\widetilde{{X}}_i$ are in fact spanned by
concentric combinatorial balls $\widetilde{B}_i$ in $\widetilde{{X}}$.
The covering map $f$ is then the union $\bigcup_{i\ge 1} f_i,$ where
$f_i: \widetilde{{X}}_i\rightarrow {X}$ is a locally injective
cellular map such that $f_i|_{\widetilde{{X}}_j}=f_j$, for every $j\le i$. We
denote by $\widetilde{G}_i=G(\widetilde{{X}}_i)$ the underlying graph of
$\widetilde{{X}}_i$. We denote by $\tS_i$ the set of vertices
$\tB_i\setminus \tB_{i-1}$.

Pick any vertex $v$ of $X$ as the basepoint. Define $\widetilde{B}_0=\{
\widetilde{v}\}:=\{ v\}, \widetilde{B}_1:=B_1(v,G),$ and $f_1:=$Id$_{B_1(v,G)}$.
Let $\widetilde{{X}}_1$ be the triangle-square complex spanned by $B_1(v,G).$
 Assume that, for $i\geq 1$,  we have constructed the vertex sets
$\widetilde{B}_1,\ldots,\widetilde{B}_i,$ and we have defined
the triangle-square complexes $\widetilde{{X}}_1,\ldots,\widetilde{{X}}_i$
and the corresponding cellular maps $f_1,\ldots,f_i$ from, respectively,
$\widetilde{{X}}_1,\ldots,\widetilde{{X}}_i$ to ${X}$ so
that the graph $\tG_i=G(\widetilde{X}_i)$ and the complex $\widetilde{X}_i$
satisfy
the following conditions:

\begin{enumerate}[(A{$_i$})]

\item[(P$_i$)]
$B_j(\tv,\tG_i)=\widetilde{B}_j$ for any $j\le i$;

\item[(Q$_i$)]
$\widetilde{G}_i$ satisfies the triangle and the square-pyramid conditions with respect to \ $\tv$, i.e., TC($\tv$) and SPC($\tv$).

\item[(R$_i$)]
for any $\widetilde{u}\in \widetilde{B}_{i-1},$ $f_i$
defines an isomorphism between the subgraph of $\tG_i$ induced by
$B_1(\widetilde{u},\tG_i)$ and the subgraph of $G$ induced by
$B_1(f_i(\widetilde{u}),G)$;

\item[(S$_i$)]
for any $\widetilde{w},\widetilde{w}'\in \widetilde{B}_{i-1}$ such that the
vertices $w=f_i(\widetilde{w}),w'=f_i(\widetilde{w}')$ belong to a square
$ww'uu'$ of ${\bX}$, there exist $\widetilde{u},\widetilde{u}'\in
\widetilde{B}_i$ such that $f_i(\widetilde{u})=u, f_i(\widetilde{u}')=u'$ and
$\widetilde{w}\widetilde{w}'\widetilde{u}\widetilde{u}'$ is a square of
$\widetilde{{X}}_i$.

\item[(T$_i$)]
 for any $\widetilde{w}\in \widetilde{S}_i:=\widetilde{B}_i\setminus
  \widetilde{B}_{i-1},$ $f_i$
  defines an isomorphism between the subgraphs of $\tG_i$ and of $G$ induced by
$B_1(\widetilde{w},\tG_i)$ and
  $f_i(B_1(\widetilde{w},\tG_i))$.

\item[(U$_i$)]
$\widetilde{G}_i$ satisfies the positioning condition with respect to $\tv$.

\end{enumerate}

 It can be easily checked that, $\widetilde{B}_1,\widetilde{G}_1,\widetilde{X}_1$ and $f_1$ satisfy the
conditions (P$_1$) through (U$_1$). Now we construct the set $\widetilde{B}_{i+1},$ the graph
$\widetilde{G}_{i+1}$ having $\widetilde{B}_{i+1}$ as the vertex-set, the
triangle-square complex $\widetilde{{X}}_{i+1}$ having $\tG_{i+1}$ as its
$1$--skeleton,  and the map $f_{i+1}: \widetilde{{X}}_{i+1}\rightarrow {X}$.
Let
 $$Z=\{ (\widetilde{w},z): \widetilde{w} \in \widetilde{S}_i \mbox{ and } z\in
B_1(f_i(\widetilde{w}),G)\setminus f_i(B_1(\widetilde{w},\tG_i))\}.$$
On $Z$ we define a binary relation $\equiv$ by setting $(\widetilde{w},z)\equiv
(\widetilde{w}',z')$ if and only if $z=z'$  and one of the following three
conditions is satisfied:

\begin{itemize}
\item[(Z1)] $\widetilde{w}$ and $\widetilde{w}'$ are the same or adjacent in
$\widetilde{G}_i$;
\item[(Z2)] there exists $\widetilde{u}\in \widetilde{B}_{i-1}$ adjacent in
$\widetilde{G}_i$ to $\widetilde{w}$ and $\widetilde{w}'$ and such that
$f_i(\widetilde{u})f_i(\widetilde{w})zf_i(\widetilde{w}')$ is a square
in $G$;
\item[(Z3)] there exists a square in $\widetilde{S}_i$ containing $\widetilde{w}$ and $\widetilde{w}'$ such that its image under $f_i$ together with $z$ induces a pyramid in $G$.
\end{itemize}

In what follows, the above relation will be used in the inductive step to construct $\tX_{i+1}$, $f_{i+1},$ and all related objects.

\subsection{Definition of $\tG_{i+1}$.}
\label{s:def-i+1}
In this subsection, performing the inductive step, we define $\tG_{i+1}$ and $f_{i+1}$.
First however we show that the relation $\equiv$ defined in the previous subsection is an equivalence relation.
The set of vertices of the graph $\tG_{i+1}$ will be then defined as the union of the set of vertices of the previously constructed graph $\tG_{i}$ and the set of equivalence classes of $\equiv$.
\medskip

\noindent
{\bf Convention:} In what follows, for any vertex $\widetilde{w}\in \widetilde{B}_i,$ we will
  denote by $w$ its image $f_i(\widetilde{w})$ in $X$.
\medskip

We now aim at showing that the relation $\equiv$ is an equivalence relation (Proposition \ref{equiv}). First we prove two auxiliary results.

\begin{lemma}
\label{l:a1a2a3}
For any couple $(\tw,z) \in Z$ the following properties hold:
\begin{enumerate}
\item[($A_1$)] there is no neighbor $\tz \in
  \tB_{i}$ of $\tw$ such that $f_i(\tz)=z$;
\item [($A_2$)] there is no neighbor  $\tu \in
  \tB_{i-1}$ of $\tw$ such that $u \sim z$;
\item [($A_3$)] there are no $\tx, \ty \in \tB_{i-1}$ such that
  $\tx \sim \tw, \ty$ and   $y \sim z$.
\end{enumerate}
\end{lemma}

\begin{proof}
If $\tw$ has a neighbor $\tz \in \tB_{i-1}$ such that $f_i(\tz)=z$,
then $(\tw,z) \notin Z$, a contradiction. This establishes ($A_1$).

If $\tw$ has a neighbor $\tu \in \tB_{i-1}$ such that $u \sim z$, then
by (R$_i$) applied to $\tu$, there exists $\tz \in \tB_i$ such that
$\tz \sim \tu,\tw$ and $f_i(\tz)=z$. Thus $(\tw,z) \notin Z$, a contradiction, establishing ($A_2$).

If there exist $\tx, \ty \in \tB_{i-1}$ such that
$\tx \sim \tw, \ty$ and $y \sim z$, then $yxwz$ is a
 square in $G$. From (S$_i$) applied to $\ty,\tx$, there
exists $\tz \in \tB_i$ such that $\tz \sim \ty,\tw$ and $f_i(\tz)=z$. Thus
$(\tw,z)
\notin Z$, a contradiction, and therefore ($A_3$) holds as well.
\end{proof}

\begin{lemma}\label{c:xexists}
  Let $\tu,\tu'\in \wt B_{i-1}$ and $\tw,\tw',\tw'' \in \tS_{i}$ be such that
  $\tu \sim \tw,\tw'$ and $\tu' \sim \tw',\tw''$. If $\tw \sim \tw'$,
  then there exist $\ty \in \wt B_{i-1}$ and $\tx \in \wt B_{i-2}$ such
  that $\ty \sim \tw, \tw'$ and $\tx \sim \tu', \ty$.
\end{lemma}

\begin{proof}
If there exists $\tx \in \wt B_{i-2}$ such that $\tx \sim \tu, \tu'$,
we are done by setting $\ty = \tu$. Assume in the following that it is
not the case. By the square-pyramid condition (Q$_i$), there exist
$\ty, \ty' \in \wt B_{i-1}$ such that $\tw' \sim \ty, \ty'$ and
$\tu,\ty,\tu',\ty'$ is a square. By (T$_i$) applied to
$\tw'$, $uyu'y'$ is a square. Consider the triangles
$uw'y,uw'y',$ and $uw'w$, all three sharing the common edge $uw'$. Since
$G$ does not contain propellers (cf.\ Lemma \ref{l:noprop}), either
$w\sim y$ or $w\sim y'$, say $w\sim y$. By (R$_i$) applied to $\tu$,
we get $\tw \sim \ty$. Using the triangle condition (Q$_i$), we get a
vertex $\tx \in \wt B_{i-2}$ such that $\tx \sim \ty,\tu'$.
\end{proof}

\begin{proposition} \label{equiv} The relation $\equiv$ is an equivalence relation on
$Z$.
\end{proposition}
\begin{proof} Since the
  binary relation $\equiv$ is reflexive and symmetric, it suffices to
  show that $\equiv$ is transitive. Let $(\widetilde{w},z)\equiv
  (\widetilde{w}',z')$ and $(\widetilde{w}',z')\equiv (\widetilde{w}'',z'')$. We
  will prove that $(\widetilde{w},z)\equiv (\widetilde{w}'',z'')$. By
  the definition of $\equiv,$ we conclude that $z=z'=z''$. By the definition of
$\equiv$, we have $z \sim w,w',w''$.

If $\tw \sim \tw''$ (in $\widetilde{G}_i$), then by the definition of $\equiv$, we have
$(\tw,z)\equiv
(\tw'',z)$ and we are done. If $\tw \nsim \tw''$ and if there exists
$\tu \in \tB_{i-1}$ such that $\tu \sim \tw, \tw''$, then by (R$_i$)
applied to $\tu$, we obtain that $u \sim w, w''$ and $w \nsim
w''$. Since $(\tw,z), (\tw'',z) \in Z$, we have $z\sim w,w''$. By ($A_2$) (cf.\ Lemma \ref{l:a1a2a3}) we have that $z \nsim u$. Thus $uwzw''$ is a
square in $G$,
and by condition (Z2), we are done. Therefore, in the rest of the
proof, we will assume the following:
\begin{enumerate}
\item[($A_4$)] $\tw \nsim \tw''$;
\item[($A_5$)] there is no $\tu \in \tS_{i-1}$ such that $\tu \sim
\tw, \tw''$.
\end{enumerate}
Observe that it implies in particular that $i\geq 2$.

\begin{claim}
  \label{c:xz=3}
  Let $\tu,\tu'\in \wt B_{i-1}$ be two vertices with $\tu \sim \tw,\tw'$ and $\tu' \sim \tw',\tw''$. If $\tx \in \wt B_{i-2}$ is adjacent to both $\tu$ and $\tu'$, then $d(x,w)=d(x,w')=d(x,w'')=2$ and $d(x,z)=3$.
\end{claim}
\begin{proof}
By the condition (R$_i$) applied to $\tu$ (respectively, $\tu'$) we have that $d(x,w)=d(x,w')=2$ (respectively, $d(x,w'')=d(x,w')=2$). We show now that $d(x,z)=3$. By ($A_2$) we have that $x\neq z$, and by ($A_3$) we have that $x\nsim z$. Assume that $d(x,z)=2$. By the local triangle condition, there exists a vertex $x'\sim z,w,x$. If $x'\sim u$ then, by (R$_i$), there exists a vertex $\wt x' \in \wt B_{i-1}$ such that $\wt x' \sim \tw,\tu,\tx$ and $f_i(\wt x')=x'$. This however contradicts ($A_2$). If $x'\nsim u$, then consider the square $x'wux$. By (S$_i$) applied to vertices $x,u$, there exists a square $\wt x'\wt w_0 \tu\tx$ in $\wt X_i$ with $f_i(\wt x')=x'$ and $f_i(\wt w_0)=w_0$. By (R$_i$) applied to $\tu$, we have that $\wt w_0 =\tw$.
Again we obtain $\wt x' \sim \tw, \tx$ and $x'=f_i(\wt x')\sim z$, which contradicts ($A_2$).
In any case we get a contradiction, thus we must have $d(x,z)=3$.
\end{proof}

We distinguish six cases depending on which of the conditions (Z1), (Z2), or (Z3) are satisfied
by the pairs $(\widetilde{w},z)\equiv
(\widetilde{w}',z')$ and $(\widetilde{w}',z')\equiv (\widetilde{w}'',z'')$.

\medskip\noindent
{\bf Case (Z1)(Z1):}  $\widetilde{w}'$ is adjacent in  $\widetilde{G}_i$ to both
$\widetilde{w}$ and $\widetilde{w}''$.

\medskip
By (T$_i$), we have that $w\neq w''$ and $w\nsim w''$. By (Q$_i$), the
graph $\widetilde{G}_i$ satisfies the triangle condition
TC($\widetilde{v}$), thus there exist two vertices
$\widetilde{u},\widetilde{u}'\in \widetilde{S}_{i-1}$ such that
$\widetilde{u}$ is adjacent to $\widetilde{w}, \widetilde{w}'$ and
$\widetilde{u}'$ is adjacent to $\widetilde{w}',\widetilde{w}''$. By
($A_5$), we have $\tu \nsim \tw''$, $\tu' \nsim \tw$, in particular
$\tu \neq \tu'$. By (R$_i$) applied to $\tw'$, it implies that $u
\nsim w''$ and $u' \nsim w$.

By Lemma~\ref{c:xexists}, we can assume that there exists a vertex
$\wt x \in \wt B_{i-2}$ adjacent to both $\wt u$ and $\wt u'$.  By
Claim \ref{c:xz=3}, we have $d(x,w)=d(x,w')=d(x,w'')=2$ and
$d(x,z)=3$.  By the interval condition applied to $I(w,w'')$, either
there exists a vertex $u_0\sim w,w',w''$ with $u_0\nsim z$, or there
exists a vertex $w'''\sim z,w,w''$ with $w'''\nsim w'$.  In the first
case, by the local positioning condition we have that $u_0\sim x$.
Observe that $u_0 \sim u$ (respectively, $u_0 \sim u'$), since otherwise
$x,w,w'$ (respectively, $x,w',w''$) belong to the interval $I(u,u_0)$ (respectively, to $I(u',u_0)$),
and $x \nsim w,w'$ (respectively, $x \nsim w,w'$), contradicting the interval condition.  By (R$_i$) applied to
$\tu$, there is a vertex $\wt u_0 \sim \tw,\tw',\tu,\tx$ with
$f_i(\tu_0) = u_0$.  By (R$_i$) applied to $\tx$, $\tu_0 \sim \tu'$
and by (R$_i$) applied to $\tu'$, $\tu_0 \sim \tw''$. This contradicts
($A_5$).

Thus, there exists a vertex $w'''\sim z,w,w''$ with
$w'''\nsim w'$.  By the local positioning condition, we have
$d(x,w''')=2$. By the local triangle condition applied to the edge
$w''w'''$, there is a vertex $u''\sim w'',w''',x$. Since $u \nsim
w''$, $u \neq u''$.  If $u''=u'$, then by the interval condition applied to
$I(w',w''')$, we have $u'' \sim w$. Thus, by (T$_i$) applied to $\tw'$, we get $\tu' \sim \tw$, contradicting ($A_5$).  If $u''\sim u'$,
by (R$_i$) applied to $\tu'$ there is a vertex $\wt u'' \in \wt
B_i$ with $\wt u'' \sim \tx,\tu', \tw''$, and $f_i(\wt u'')=u''$. If
$u''\nsim u'$, by (S$_i$) applied to the square $xu'w''u''$ and
to the vertices $\tu',\tx$, there is a vertex $\wt u'' \in \wt B_i$
with $\wt u'' \sim \tx, \tw''$ and $f_i(\wt u'')=u''$. In both cases,
applying (R$_i$) to $\wt u''$ we get a vertex $\wt w'''\sim \tw '',
\wt u''$ such that $f_i(\wt w ''')=w'''$ and $d(\tx,\tw''')=2$.

Proceeding analogically for the edge $ww'''$ (instead of $w''w'''$) we
obtain a vertex $\wt u''' \sim \tw,\tx$ such that $u'''=f_i(\wt
u''')\sim w,w''',x$. If $u'' \sim u'''$, by (R$_i$) applied to
$\tx$, we have $\tu'' \sim \tu'''$, and by (R$_i$) applied to $\tu''$ we get that $\tu'''\sim \tw'''$. If $u'' \nsim u'''$, then again we obtain that $\tu''' \sim \tw'''$, by using (S$_i$) for the square $xu''w'''u'''$ and (R$_i$) for $\tx$ and $\tu''$.
By (R$_i$) applied to $\tu'''$, we have $\tw \sim \tw'''$, and by (T$_i$) applied to $\tw$, we have $\tw'''\nsim \tw'$. Consequently, the condition (Z3) holds
for $\tw,\tw''$ and the pyramid $zww'w''w'''$. Hence, $(\tw,z) \equiv (\tw'',z'')$.
This finishes the proof in Case (Z1)(Z1).

\medskip\noindent
{\bf Case (Z1)(Z2):}  $\widetilde{w}$ is adjacent in  $\widetilde{G}_i$ to $\widetilde{w}'$, and there exists $\tu'\in \widetilde{B}_{i-1}$ adjacent in
$\widetilde{G}_i$ to $\widetilde{w}'$ and $\widetilde{w}''$ and such that
$w'u'w''z$ is a square
in $G$.

\medskip
By (R$_i$), we have  $w\neq w''$. By (Q$_i$), there exists a vertex
$\widetilde{u}\in \widetilde{S}_{i-1}$ such that $\widetilde{u}\sim \tw,\tw'$. By ($A_5$), we have $\tu \nsim \tw''$, $\tu' \nsim
\tw$, in particular, $\tu \neq \tu'$.

By Lemma~\ref{c:xexists}, we can assume that there exists a vertex
$\wt x \in \wt B_{i-2}$ adjacent to both $\wt u$ and $\wt u'$.
By Claim \ref{c:xz=3}, we have $d(x,w)=d(x,w')=d(x,w'')=2$ and $d(x,z)=3$. Note that $w\nsim w''$, otherwise the square $ww'u'w''$ would falsify the local positioning
condition for $x$. By the local triangle condition there is a vertex $y\sim w'',u',w$.

Assume first that $y\sim z$. Then $d(x,y)=2$, and consequently $y\sim w'$ (otherwise the square $yww'u'$ falsifies the local positioning condition for $x$). The condition (R$_i$) applied to $\tu'$ shows that there exists a vertex $\ty \sim \tu',\tw', \tw''$ with $f_i(\ty)=y$. By ($A_2$), we have $\ty \in \wt S_i$, and by (T$_i$) applied to $\tw'$, we have that $\ty \sim \tw$. Hence the situation is the same as in Case (Z1)(Z1), with $\ty$ playing the role of $\tw'$.

Assume now that $y\nsim z$. By the local positioning
condition applied to the square $wyw''z$ and the basepoint $x$, we have
that $y\sim x$. By (R$_i$) applied to $\tu'$, there is a vertex $\ty
\sim \tx,\tw'',\tu'$ in $\wt S_{i-1}$ with $f_i(\ty)=y$. If $y\sim u$, by
(R$_i$) (applied to $\tx$ and then to $\tu$), we have that $\ty\sim
\tw$.  If $y\nsim u$,  by (S$_i$) applied to the square $xywu$ and
to the vertices $\tx,\tu$, we also get $\ty \sim \tw$.  Since, $\ty\in \wt
S_{i-1},$ applying (Z2) to the square $wzw''y$, we obtain that
$(\tw,z)\equiv (\tw'',z'')$. This finishes the proof in Case (Z1)(Z2).

\medskip\noindent
{\bf Case (Z1)(Z3):}  $\tw$ is adjacent in  $\tG_i$ to $\tw'$ and there exist $\tu',\tu''\in \tS_{i}$ adjacent in
$\widetilde{G}_i$ to $\tw',\tw''$ and such that the vertices
$u',u'',w',w'',z$ induce a pyramid in $G$.

\medskip
By (T$_i$), $w \neq u', w \neq u''$ and $w \neq w''$.  By the
no-propeller Lemma \ref{l:noprop} applied to the triangles
$w'zw,w'zu',$ and $w'zu''$, either $w\sim u'$ or $w\sim u'',$ say
$w\sim u'$. By the condition (T$_i$), $\tw\sim \tu'$. Then
replacing $\tw'$ by $\tu',$ since $(\tw,z)\equiv(\tu',z),
(\tu',z)\equiv (\tw'',z)$ and $\tw\sim \tu'\sim\tw''$, we are in
conditions of Case (Z1)(Z1), thus $(\tw,z)\equiv (\tw'',z)$ and we are
done.

\medskip\noindent
{\bf Case (Z2)(Z2):}  There exists $\tu\in \widetilde{B}_{i-1}$ adjacent in
$\widetilde{G}_i$ to $\tw$ and $\tw'$ and there exists $\tu'\in \widetilde{B}_{i-1}$ adjacent in
$\widetilde{G}_i$ to $\tw'$ and $\tw''$ such that $wuw'z$ and $w'u'w''z$ are  squares
in $G$.

\medskip
By ($A_5$), $\tu\ne \tu'$ and $\tu\nsim \tw'', \tu'\nsim \tw$.
By (T$_i$), we have $u\ne u'$ and $w \nsim w'$, $w' \nsim w''$. If $w
= w''$, then $u, u', z$ belong to the interval $I(w,w')$
and consequently, $z \sim u$ or $z \sim u'$, contradicting ($A_2$).
If $u \sim w''$, then $w, w', w''$ belong to the interval $I(u,z)$;
consequently, either $w' \sim w$ or $w' \sim w''$, a
contradiction. For the same reasons, $u'\nsim w$.

If $u\sim u',$ since $G$ does not contain half open books (Lemma
\ref{l:noprop}), the previous constraints imply that $w\sim w''$.
Then (S$_i$) applied to the square $wuu'w''$ implies that $\tw\sim
\tw''$ and we are done because $(\tw,z)\equiv (\tw'',z)$ by (Z1). So,
further we will suppose that $u\nsim u'$. By (T$_i$), $\tu\nsim
\tu'$. Below, we consider separately two cases, (i) and (ii).
\medskip

\noindent
(i) There exists a vertex $\tx \in \tB_{i-2}$ adjacent to both $\tu$ and $\tu'$.
\medskip

\noindent
By Claim \ref{c:xz=3}, we have $d(x,w)=d(x,w')=d(x,w'')=2$ and $d(x,z)=3$. By the interval condition, the vertices $w$ and $w''$ belong to a square of $G$. If this square contains $z,$ then the fourth vertex of this square, denote it $y$, will be adjacent to $x$ by the local positioning condition. By (R$_i$) applied to $\tx$, there exists a vertex $\ty\in \tB_{i-1}$ with $\ty\sim\tx$ and $f_i(\ty)=y$. If $y\sim u,$ then by (R$_i$) applied to $\tx$ and then to $\tu$ we obtain that $\tu\sim\ty$ and $\ty\sim\tw$. On the other hand, if $u\nsim y,$ then (S$_i$) applied to the square $xywu$ also implies that $\ty\sim \tw$. Analogously, we can conclude that $\ty\sim \tw''$.
By (Z2) applied to the square $wyw''z$, we deduce that $(\tw,z)\equiv (\tw'',z),$ as required.

 Now suppose that any square of $G$ containing $w$ and $w''$ has the form $wyw''y',$ where $y,y'\ne z$ and $z\sim y,y'$. By the local positioning condition and since $d(x,z)=3$, we have $d(x,y)=d(x,y')=2$.
 By the local triangle condition applied to vertices $w,y,x$ (respectively, $w'',y,x$) there exists a vertex $s\sim w,y,x$ (respectively,  $s'\sim w'',y,x$). By (R$_i$) applied to $\tx$, there exist vertices $\ts,\ts' \in \tB_i$ adjacent to $\tx$ and such that
 $f_i(\ts)=s, f_i(\ts')=s'$. By (R$_i$) applied to $\tx$ and then to $\tu$ (if $s\sim
 u$) or by (S$_i$) (applied to the square $uxsw$, if $s\nsim u$), we
 have  $\ts \sim \tw$. Similarly, $\ts'\sim
 \tw''$. Again, by (R$_i$) applied to $\ts$, there is a vertex $\ty\sim
 \ts,\tw$ with $f_i(\ty)=y$. By (R$_i$) applied to $\ts$ (if $s\sim
 s'$) or by (S$_i$) (applied to the square $ysxs'$ if $s\nsim s'$), we have that $\ty \sim \ts'$.
 Then, by (R$_i$) applied to $\ts'$, we obtain that $\ty \sim \tw''$. Analogously, we show that there is a vertex $\ty' \sim \tw,\tw''$ with $f_i(\ty')=y'$.
 Since $z\sim y,y'$, by ($A_2$) we have that $\ty,\ty' \in \tS_i$.
 As a consequence, $(\tw,z)\equiv (\tw'',z)$ by (Z3).

\medskip
\noindent
(ii) There is no vertex in $\wt B_{i-2}$ adjacent to both $\wt u$ and $\wt u'$.
\medskip

\noindent
By the square-pyramid condition (Q$_i$), there exist two distinct
vertices $\ty,\ty'\in \wt S_{i-1}$ with $f_i(\ty)=y,f_i(\ty')=y'$,
such that $\tw' \sim \ty,\ty'$ and $\tu\ty\tu'\ty'$ is a square.
By (R$_i$), the vertices $y,y'$ are both adjacent to $u,u',w'$, and $y\nsim y'$. By the triangle condition (Q$_i$) there is a
vertex $\tx \in \tB_{i-2}$ adjacent to $\ty$ and $\tu'$. If $\ty \sim
\tw$, then replacing $\tu$ by $\ty$ and applying case (i), we are
done.

Suppose now that $\ty \nsim \tw$.  By ($A_2$), we have $z \nsim y,y'$. By the
local triangle condition, in $G$ there exists a common neighbor $w'''$
of $y,w,$ and $z$. If $w'''\sim u$, by (R$_i$) applied to $\tu$, there
exists a vertex $\tw''' \sim \tu,\tw,\ty$ with $f_i(\tw''')=w'''$. If
$w''' \nsim u$, by (S$_i$) applied to the square $wuyw'''$, there
exists $\tw''' \sim \tw,\ty$ with $f_i(\tw''')=w'''$.  Since $w'''\sim
z$, by ($A_2$), we have that $\tw'''\in \tS_i$, and by (T$_i$) and
($A_1$), we have $(\tw''',z)\in Z$.  Thus, if $\tw ''' \nsim \tw'$, by
the preceding case (i) (since, by the triangle condition, there is a vertex $\tx \in \tB_{i-2}$ adjacent to $\ty$ and $\tu'$) we have that $(\tw''',z)\equiv (\tw'',z)$.  If
$\tw ''' \sim \tw'$, we get the same conclusion by applying Case
(Z1)(Z2).  Since $\tw\sim \tw'''$, we have $(\tw,z)\equiv (\tw''',z)$
by condition (Z1). Hence, we obtain that $(\tw,z)\equiv (\tw'',z)$ by
applying one of the cases (Z1)(Z1), (Z1)(Z2), or (Z1)(Z3) to the
triplet $(\tw,z),(\tw''',z),(\tw'',z)$. This finishes the proof in
Case (Z2)(Z2).

\medskip\noindent
{\bf Case (Z2)(Z3):} There exists $\tu\in \widetilde{B}_{i-1}$ adjacent in
$\tG_i$ to $\tw,\tw'$, such that $wuw'z$ is a square, and there exist $\tu',\tu''\in \tS_{i}$ adjacent in
$\tG_i$ to $\tw',\tw''$ and such that the vertices
$u',u'',w',w'',z$ induce a pyramid in $G$.

\medskip\noindent By (T$_i$), $u \notin \{u',u'',w''\}$ and, by (R$_i$) $w \notin
\{u',u''\}$. If $w = w''$, then by the interval condition for $I(w,w')$,
we have $u \sim u', u''$. Consequently, by (T$_i$) applied to $\tw'$ and to
$\tu'$, $\tu \sim \tu'$ and $\tu \sim \tw''$,
which contradicts ($A_5$). Thus $w\neq w''$. Moreover, $w \nsim w'$, $w'\nsim w''$,
$u'\nsim u''$, and $u\nsim z$.  Hence, by (T$_i$) and ($A_1$), we have $(\tu',z) \in Z$.  Applying Case (Z1)(Z2) to the triplet
$(\tu',z),(\tw',z),(\tw,z)$, we obtain that $(\tw,z)\equiv (\tu',z)$.
By applying one of the cases (Z1)(Z1), (Z1)(Z2), or (Z1)(Z3) to the
triplet $(\tw'',z),(\tu',z),(\tw,z)$, we conclude that $(\tw,z)\equiv
(\tw'',z)$. This finishes the proof in Case (Z2)(Z3).

\medskip\noindent
{\bf Case (Z3)(Z3):} There exist $\ty',\ty''\in \tS_i$ adjacent in
$\widetilde{G}_i$ to $\tw,\tw'$ and such that the vertices
$y',y'',w,w',z$ induce a pyramid in $G$, and there exist $\tu',\tu''\in \widetilde{S}_{i}$ adjacent in
$\widetilde{G}_i$ to $\tw',\tw''$ and such that the vertices
$u',u'',w',w'',z$ induce a pyramid in $G$.

\medskip\noindent Again, we first notice that $w\neq w', w'\neq w''$,
$w\nsim w',$ and $w'\nsim w''$. Similarly $y'\ne y'', u'\ne u''$,
$y'\nsim y'',$ and $u'\nsim u''$. By (T$_i$) and ($A_1$), we have $(\tu,z),
(\tu',z), (\ty,z), (\ty',z) \in Z$.  First suppose that one of the
vertices $y',y''$ coincides with one of the vertices $u',u''$, say
$y'=u'$. Then, by (T$_i$) applied to $\tw'$, we have $\ty'=\tu'$ and
$(\tw,z)\equiv (\tu',z)$ and $(\tu',z)\equiv (\tw'',z)$ by
(Z1). Consequently, Case (Z1)(Z1) implies that $(\tw,z)\equiv
(\tw'',z)$. Thus suppose that the vertices $y',y''$ and $u',u''$ are
pairwise distinct. By Case (Z1)(Z3) applied to the triplet
$(\tu',z),(\tw',z),(\tw,z)$ we deduce that $(\tw,z)\equiv
(\tu',z)$. Then applying one of the cases (Z1)(Z1), (Z1)(Z2), or
(Z1)(Z3) to the triplet $(\tw'',z),(\tu',z),(\tw,z)$, we obtain that
$(\tw,z)\equiv (\tw'',z)$. This finishes the proof of the Case
(Z3)(Z3) and completes the proof that $\equiv$ is an equivalence
relation on $Z$.
\end{proof}

Let $\widetilde{S}_{i+1}$ denote the set of equivalence classes of $\equiv$, i.e.,
$\widetilde{S}_{i+1}=Z/_{\equiv}$. For a couple $(\widetilde{w},z)\in Z$,
we will denote by $[\widetilde{w},z]$ the equivalence class of $\equiv$
containing  $(\widetilde{w},z)$. Set $\widetilde{B}_{i+1}:=\widetilde{B}_i\cup
\widetilde{S}_{i+1}$.
Let $\widetilde{G}_{i+1}$ be the graph having  $\widetilde{B}_{i+1}$ as the
vertex set in which two vertices $\widetilde{a},\widetilde{b}$ are adjacent if
and
only if one of the following conditions holds:
\begin{itemize}
\item[(1)] $\widetilde{a},\widetilde{b}\in \widetilde{B}_i$ and
$\widetilde{a}\widetilde{b}$ is an edge of $\widetilde{G}_i$,
\item[(2)] $\widetilde{a}\in \widetilde{B}_i$,  $\widetilde{b}\in
\widetilde{S}_{i+1}$ and $\widetilde{b}=[\widetilde{a},z]$,
\item[(3)] $\widetilde{a},\widetilde{b}\in \widetilde{S}_{i+1},$
$\widetilde{a}=[\widetilde{w},z]$,
$\widetilde{b}=[\widetilde{w},z']$ for a vertex $\widetilde{w}\in \tB_i,$ and
$z\sim z'$ in the graph $G$.
\end{itemize}

Finally, we define the map $f_{i+1}: \widetilde{B}_{i+1}\rightarrow V(G)$ in
the
following way: if $\widetilde{a}\in \widetilde{B}_i$, then
$f_{i+1}(\widetilde{a})=f_i(\widetilde{a}),$  otherwise, if $\widetilde{a}\in
\widetilde{S}_{i+1}$ and
$\widetilde{a}=[\widetilde{w},z],$ then $f_{i+1}(\widetilde{a})=z$. Notice that
$f_{i+1}$ is well-defined because all couples from the equivalence
class representing $\widetilde{a}$ have one and the same vertex $z$ in the
second argument. In the sequel we follow our earlier convention for notations: all vertices of $\widetilde{B}_{i+1}$ will
be denoted with the tilde and their images in $G$ under $f_{i+1}$ will be denoted
without a tilde, e.g.\ if $\widetilde{w}\in \widetilde{B}_{i+1},$
then $w=f_{i+1}(\widetilde{w})$.

\subsection{Properties of $\tG_{i+1}$ and $f_{i+1}$.}
\label{s:propG}
In this subsection we check our inductive assumptions, verifying the properties (P$_{i+1}$) through (U$_{i+1}$) for $\tG_{i+1}$ and $f_{i+1}$ defined above. In particular it allows us to define the corresponding complex $\tX_{i+1}$.
\medskip

\begin{lemma}[Property (P$_{i+1}$)] \label{Pi+1}  The graph $\tG_{i+1}$ satisfies the property $(P_{i+1})$,
i.e.,
$B_j(v,\tG_{i+1})=\widetilde{B}_j$ for any $j\le i+1$.
\end{lemma}

\begin{proof} By the definition of edges of $\widetilde{G}_{i+1},$ any vertex
$\widetilde{b}$ of $\widetilde{S}_{i+1}$ is adjacent
to at least one vertex of $\tB_i$ and all such neighbors of $\widetilde{b}$ are
vertices of the form $\widetilde{w}\in \widetilde{B}_i$
such that $\widetilde{b}=[\widetilde{w},z]$ for a couple $(\widetilde{w},z)$ of
$Z$. By the definition of $Z,$ $\widetilde{w}\in \tS_i,$
whence any vertex of $\tS_{i+1}$ is adjacent only to vertices of $\tS_i$ and
$\tS_{i+1}$. Therefore, the distance between
the basepoint $\widetilde{v}$ and any vertex $\widetilde{a}\in \tB_i$ is the
same in the graphs $\tG_i$ and $\tG_{i+1}$. On the
other hand, the distance in  $\tG_{i+1}$ between $\widetilde{v}$ and any vertex
$\widetilde{b}$ of $\tS_{i+1}$ is $i+1$. This
shows that indeed $B_j(v,\tG_{i+1})=\widetilde{B}_j$ for any $j\le i+1$.
\end{proof}

\begin{lemma}[Property (Q$_{i+1}$)] \label{Qi+1}  The graph $\widetilde{G}_{i+1}$ satisfies the property
$(Q_{i+1})$, i.e.\ $\widetilde{G}_{i+1}$
satisfies the triangle and the square-pyramid conditions with respect to $\tv$.
\end{lemma}

\begin{proof} First we show that  $\widetilde{G}_{i+1}$ satisfies the triangle
condition TC($\widetilde{v}$).
Pick two adjacent vertices $\tu,\tw$ having in
$\widetilde{G}_{i+1}$ the same distance to $\widetilde{v}$.
Since by Lemma \ref{Pi+1}, $\widetilde{G}_{i+1}$ satisfies the property
(P$_{i+1}$) and, by (Q$_i$) the graph $\tG_i$ satisfies the triangle condition with respect to $\tv$, we can suppose that
$\tu,\tw \in \tS_{i+1}$. From the definition of the edges
of $\tG_{i+1}$, there exist two couples $(\tx,z),(\tx,z')\in
Z$ such that $\tx\in \tB_i,$ $z$ is
adjacent to $z'$ in $G,$ and
$\tu=[\widetilde{x},z],\widetilde{w}=[\widetilde{x},z']$. Since
$\widetilde{x}$ is
adjacent in $\tG_{i+1}$ to both $\widetilde{u}$ and $\widetilde{w},$ the
triangle condition TC($\widetilde{v}$) is established.

Now we establish the square-pyramid condition SPC($\tv$).
Again, by Lemma \ref{Pi+1} and by (Q$_i$) for $\tG_i$, it is enough to consider the situation when $\tu \in \tS_{i+1}$ and $\tu$ is adjacent to mutually non-adjacent $\tw,\tw' \in \tS_i$. By the definition of $\tS_{i+1}$ we have that $\tu=[\tw,z]=[\tw',z]$.
Thus $(\tw,z)\equiv (\tw',z)$ and, by the definition of the equivalence relation $\equiv$, we have the following two cases.
\medskip

\noindent
\emph{Case 1:} There exists a vertex $\tx\in \tB_{i-1}$ adjacent to $\tw,\tw'$. Then we obtain the square $\tu\tw\tx\tw'$ as required by the square-pyramid condition with respect to $\tv$.

\medskip

\noindent
\emph{Case 2:} There exists a square $\tw\tx\tw'\tx'$ in $\tS_i$ such
that vertices $w,x,w',x',z$ induce a pyramid in $G$. Observe that by
(R$_i$) and (T$_i$), we obtain $(\tx,z), (\tx',z) \in Z$, and by the definition
of $\tS_{i+1}$ we have $\tu=[\tw,z]=[\tx,z]=[\tx',z]$ and thus $\tu
\sim \tx,\tx'$. Moreover, by (T$_i$) applied to $\tw$, we obtain that
$\tx \nsim \tx'$. Hence the square-pyramid condition is verified.
\end{proof}

Now we establish some properties of the map $f_{i+1}$.
We first prove that the mapping $f_{i+1}$ is a graph homomorphism
(preserving edges) from
$\tG_{i+1}$ to $G$. In particular, this implies that two adjacent vertices of
$\tG_{i+1}$ are mapped in $G$ to
different vertices.

\begin{lemma}\label{lem-homomorphism}
$f_{i+1}$ is a graph homomorphism from $\tG_{i+1}$ to $G$, i.e., for any edge
$\ta\tb$ of $\tG_{i+1}$, $ab$ is an edge of $G$.
\end{lemma}

\begin{proof}
Consider an edge $\ta\tb$ of $\tG_{i+1}$. If $\ta,\tb \in \tB_i$, the
lemma holds by (R$_i$) or (T$_i$)  applied to $\ta$. Suppose that $\ta \in
\tS_{i+1}$. If $\tb \in \tB_i$, then $\ta = [\tb,a]$, and $ab$ is an
edge of $G$. If $\tb \in \tB_{i+1}$, then the fact that $\ta$ and $\tb$ are
adjacent  implies that there
exists a vertex $\tw \in \tB_i$ such that $\ta = [\tw,a], \tb= [\tw,b]$, and
such that $a \sim b$ in $G$.
\end{proof}
We now prove that $f_{i+1}$ is locally surjective at any vertex in
$\tB_i$.

\begin{lemma}\label{lem-locally-surjective}
If $\ta \in \tB_i$ and if $b \sim a$ in $G,$ then there exists
a vertex $\tb$ of $\tG_{i+1}$  adjacent to $\ta$ such that $f_{i+1}(\tb) = b$.
\end{lemma}

\begin{proof}
If $\ta \in \tB_{i-1}$, the lemma holds by (R$_i$). Suppose that $\ta
\in \tS_i$ and consider $b \sim a$ in $G$. If $\ta$ has a
neighbor $\tb \in \tB_i$ mapped to $b$ by $f_i$, we are
done. Otherwise $(\ta,b) \in Z$, $[\ta,b]\sim \ta $ in
$\tG_{i+1}$ and $[\ta,b]$ is mapped to $b$ by $f_{i+1}$.
\end{proof}

We now prove that $f_{i+1}$ is locally injective.

\begin{lemma}\label{lem-locally-injective}
If $\ta \in \tB_{i+1}$ and  $\tb, \tc$ are distinct neighbors of
$\ta$ in $\tG_{i+1}$, then $b \neq c$.
\end{lemma}
\begin{proof}
  If $\ta \in \tB_{i-1}$, then the assertion of the lemma follows directly from the condition (R$_i$) applied to $\ta$. Further we proceed by contradiction, i.e., we assume that $b=c$.

  Suppose now that $\ta \in \tS_i$.
  If $\tb$ or $\tc$ is in $\tB_i$, then $(\ta,b)$ is not in $Z$. Thus, if say \ $\tb \in \tB_i$, then $\tc$ is not in $\tS_{i+1}$, otherwise we would have $\tc=[\ta,b]$.
  If $\tb,\tc \in \tB_i$, then we get a contradiction with (T$_i$) applied to $\ta$. If $\tb,\tc  \in \tS_{i+1}$, then, by the definition of vertices in $\tS_{i+1}$, we have $\tb=[\ta,b]=\tc$, contradicting the choice of $\tb,\tc$.

  Thus further we assume that $\ta \in \tS_{i+1}$. If $\tb,\tc \in \tB_i$, then $\ta=[\tb,a]=[\tc,a]$. By Lemma \ref{lem-homomorphism} we have that $\tb \nsim \tc$. Since $(\tb,a)\equiv (\tc,a)$, by the definition of the relation $\equiv$, there is a vertex $\tx \in \tB_i$ adjacent to $\tb$ and $\tc$. Then we get a contradiction by (R$_i$) or (T$_i$) applied to $\tx$. Now we suppose that $\tb \in \tS_{i+1}$ and $\tc \in \tB_i$. By the definition of the edge $\ta\tb$ there is a vertex $\td \in \tS_i$ adjacent to $\ta$ and $\tb$. Observe that $\td \nsim \tc$ since otherwise we would get a contradiction with $(\td,b)\in Z$.
  Since $(\td,a)\equiv (\tc,a)$, by the definition of the relation $\equiv$, we are either in the case (Z2) or in the case (Z3).
  Observe however that this is not possible since $f_{i+1}(\td)\sim f_{i+1}(\tc)=b$. Consequently, we obtain that it is not possible that $\tb \in \tS_{i+1}$ and $\tc \in \tB_i$.

  For the remaining part of the proof we thus suppose that
  $\ta,\tb,\tc \in \tS_{i+1}$. By the definitions of edges $\tb\ta$
  and $\ta\tc$ there exist vertices $\tw,\tw' \in \tS_i$ with $\tw
  \sim \tb,\ta$ and $\tw' \sim \ta,\tc$. Observe that $\tw \nsim \tw'$
  since otherwise we would have $\tb = [\tw,b] = [\tw',b]=\tc$. Since
  $(\tw,a)\equiv (\tw',a)$, by the definition of $\equiv$, we are in
  the case (Z2) or (Z3). For (Z2) there is a vertex $\tx \in
  \tB_{i-1}$ such that $waw'x$ is a square in $G$. By ($A_1$), $b \neq
  x$ and then $wbw'x$ is also a square in $G$; consequently, by (Z2),
  $\tb=[\tw,b] = [\tw',b] =\tc$. Thus we are in the case (Z3), i.e.,
  there exist vertices $\tx,\tx' \in \tS_i$ adjacent to $\tw,\tw'$,
  mutually non-adjacent and such that $w,w',a,x,x'$ induce a pyramid
  in $G$.  Observe that, by ($A_1$), we have $b \notin \{x,x'\}$, and
  that $b$ lies in the interval $I(w,w')$ in
  $G$. Consequently, $b \sim x,x'$, and by (Z3), $\tb
  =[\tw,b]=[\tw',b]=\tc$.
\end{proof}

Before proving the next result, we formulate two technical
lemmas.

\begin{lemma}\label{lem-aux}
If $a,b,c$ are distinct pairwise adjacent vertices of $G$ such that $\ta=[\tb,a]\in \tS_{i+1},\tb,\tc\in \tB_i$ and $\tb\sim \tc$ in $\tG_{i+1},$
then $(\tc,a)\in Z$ and, in particular, $[\tc,a]=\ta\sim \tc$.
\end{lemma}

\begin{proof} If $(\tc,a)\notin Z,$ then there exists $\ta' \sim \tc$ in
$\tB_i$ such that $f_i(\ta') = a$.  By (T$_i$) applied to $\tc$ we
  conclude that $\tb \sim \ta'$, i.e., $(\tb,a) \notin Z$, a
  contradiction. Thus $(\tc,a)\in Z$; in particular, $\ta=[\tc,a]$ by
  (Z1), and by the definition of edges of $\tG_{i+1},$ $\ta\sim \tc$.
\end{proof}

\begin{lemma}\label{lem-carre-tordu}
Let $\ta \in \tS_{i-1}$,
$\tb \in \tS_i$ and $\tc = [\tb,c] \in \tS_{i+1}$ be such that $\tb \sim
\ta,\tc$. Assume that there is a square $abcd$ in $G$.
Then there exists $\td \in \tS_i$ such that $\ta \sim \td$
and $(\td,c) \in Z$.
\end{lemma}

\begin{proof}
By (R$_i$), there exists $\td \sim \ta$ such that $f_{i}(\td)=d$. If
$\td \in \tB_{i-1}$, then by conditions (P$_i$), (R$_i$), and (S$_i$), there exists $\tc' \in \tB_i$
such that $\tc' \sim \tb, \td$ and $f_{i}(\tc') = c$; this is impossible since $(\tb,c) \in
Z$. Consequently, $\td \in \tS_i$.

Suppose now that $(\td,c) \notin Z$. It means that there exists $\tc'
\in \tS_{i-1} \cup \tS_i$ such that $\tc' \sim \td$ and $f_i(\tc')=
c$. We distinguish two cases depending on if $\tc' \in \tS_{i-1}$ or $\tc'
\in \tS_{i}$.

\medskip
\noindent
\textbf{Case 1: $\tc' \in \tS_{i-1}$.}
\medskip

We have $\tc' \nsim \ta$. By the triangle and square-pyramid conditions (Q$_i$), either there exists $\tx \in \tS_{i-2}$ such that
$\tx \sim \ta,\tc'$, or there exist $\ty,\ty' \in \tS_{i-1}$ such that
$\ty,\ty'\sim \ta, \tc', \td$ and $\ty \nsim \ty'$.

If there exists $\tx \in \tS_{i-2}$ such that $\tx \sim \ta,\tc'$, then
$\tx \nsim \tb, \td$. By (R$_i$), $x=f_i(\tx) \sim a,c$ and $x \nsim
b,d$. Since $b,d,x$ belong to the interval $I(a,c)$ and are
pairwise non-adjacent, we get a contradiction with the interval
condition.

Assume that there exist $\ty,\ty' \in \tS_{i-1}$ such that $\ty,\ty'\sim \ta,
\tc', \td$ and $\ty \nsim \ty'$. We have $y =
f_i(\ty) \sim a,c,d$, and by the interval condition, $y \sim b$. By
(R$_i$) applied to $\ta$ and then to $\ty$, we obtain that $\ty \sim \tb$ and $\tb
\sim \tc'$. Consequently, $(\tb,c) \notin Z$, a contradiction.

\medskip
\noindent
\textbf{Case 2: $\tc' \in \tS_{i}$.}
\medskip

By the triangle condition (Q$_i$), there exists $\tu \in \tS_{i-1}$
such that $\tu \sim \tc',\td$.  By Lemma~\ref{c:xexists}, we can
assume that there exists $\tx \in \tS_{i-2}$ such that $\tx \sim \tu,
\ta$. Note that $\tx \nsim \tb, \tc', \td$ and by (R$_i$), we have $x \nsim
b,c,d$.

If $\tu \sim \ta$, then $u \sim a$ and by the interval condition applied to the
interval $I(a,c)$, we have $u
\sim b$. By (R$_i$) applied to $\ta$ and to $\tu$, we get $\tu \sim
\tb$ and $\tb \sim \tc'$. Consequently, $(\tb,c) \notin Z$, a
contradiction.

Suppose now that $\tu \nsim \ta$. Note that if $\tu \sim \tb$, then $u
\sim b$, $u\sim a$ by the interval
condition, and $\tu \sim \ta$ by (R$_i$). Consequently, $\tu \nsim
\ta,\tb$ and $u \nsim a,b$.  Therefore, the graph induced by the
vertices $a,b,c,d, u,x$ is a half open book since $x \nsim b,c,d$,
$u\nsim a,b$, $a\nsim c$ and $b \nsim d$, a contradiction by Lemma \ref{l:noprop}.
\end{proof}

We now show that the subgraphs induced by $B_1(\ta,\tG_{i+1})$ and
$f_{i+1}(B_1(\ta,\tG_{i+1}))$ are isomorphic.

\begin{lemma}\label{lem-triangles}
Let $\ta,\tb, \tc$  be three distinct vertices in $\tG_{i+1}$ such that $\ta \sim \tb, \tc$.
Then $\tb \sim \tc$ if and only
if $b \sim c$.
\end{lemma}

\begin{proof}
By Lemma \ref{lem-locally-injective}, $b\ne c$. If $\ta, \tb, \tc \in
\tB_i$, then the lemma holds by the condition (T$_i$) applied to $\ta$.
Thus further we assume that among $\ta, \tb, \tc$ there is a vertex
outside $\tB_i$.  Note that from Lemma~\ref{lem-homomorphism}, if $\tb
\sim \tc$, then $b \sim c$, establishing one direction. Suppose now
that $b \sim c$ in $G$; we will show that $\tb\sim \tc$ in
$\tG_{i+1}$.

\medskip
\noindent
\textbf{Case 1:} $\ta \in \tB_i$.
\medskip

If $\tb, \tc \in \tS_{i+1}$, then $\tb =
[\ta,b]$ and $\tc = [\ta,c]$. Since $b\sim c$, by construction, we
have $\tb \sim \tc$ in $\tG_{i+1}$. Suppose now that $\tb = [\ta,b]
\in S_{i+1}$ and $\tc \in \tB_i$. Then $\tb\sim \tc$ by Lemma \ref{lem-aux}.

\medskip
\noindent
\textbf{Case 2:} $\tb, \tc \in \tB_i$ and $\ta \in \tS_{i+1}$.
\medskip

We know that $\ta =[\tb,a] = [\tc,a]$. If $\tb\nsim \tc$, we are in
one of the cases (Z2) or (Z3) from the definition of $\equiv$.  Hence in $G$ the vertices $b$ and $c$
are opposite vertices of a square, which is impossible because $b\sim
c$.

\medskip
\noindent
\textbf{Case 3:} $\ta, \tb \in \tS_{i+1}$ and $\tc \in \tB_i$.
\medskip

Since $\tc \in \tB_i$, there exists $\tb' \in \tB_{i+1}$ such that
$\tb' \sim \tc$ and $f_{i+1}(\tb')=b$. Applying Case 1 to the triplet
$\tc, \ta, \tb'$, we get that $\tb' \sim \ta$. By
Lemma~\ref{lem-locally-injective}, we get that $\tb' = \tb$ and we are
done.

\medskip
\noindent
\textbf{Case 4:} $\ta, \tb, \tc \in \tS_{i+1}$.
\medskip

There exist $\tw, \tw' \in \tB_i$ such that $\tb = [\tw,b]$, $\tc =
[\tw',c],$ and $\ta = [\tw,a] = [\tw',a]$. If $\tw \sim \tc$ or $\tw'
\sim \tb$, then $\tb \sim \tc$ because $b\sim c$. Suppose further that
$\tw \nsim \tc$, $\tw' \nsim \tb$. From Case 3 applied to $\ta,\tb \in
\tS_{i+1}$ (respectively, $\ta,\tc \in \tS_{i+1}$) and $\tw' \in \tB_i$
(respectively, $\tw \in \tB_i$), it follows that $w \nsim c$ and $w' \nsim b$.
Since $[\tw,a]=[\tw',a],$ the vertices $\tw$ and $\tw'$ obey one of
the conditions (Z1),(Z2),(Z3).

We show that we can assume that there exists $\tx \in \tS_{i-1}$ such
that $\tx \sim \tw, \tw'$. If $(\tw,a)\equiv (\tw',a)$ by condition (Z2), we are done. If $(\tw,a)\equiv (\tw',a)$ by condition (Z1), $\tw
\sim \tw'$ and by the triangle condition, there exists $\tx \in \tS_{i-1}$
such that $\tx \sim \tw,\tw'$. Suppose now that $(\tw,a)\equiv
(\tw',a)$ by condition (Z3). Thus, there exist $\ty,\ty' \in \tS_{i}$
such that $\ty,\ty' \sim \ta,\tw,\tw'$ and $\ty \nsim \ty'$. Consider
the triangles $awy$, $awy'$ and $awb$, all three sharing the common edge
$aw$. By Lemma~\ref{l:noprop}, we get that $b$ is adjacent to $y$ or
$y'$, say $b \sim y$.
By Case 3 for $\tb,\ta,\ty$, we have $\tb \sim \ty$.  Then, we replace
$\tw$ by $\ty$, and $(\ty,a)\equiv (\tw',a)$ by condition (Z1): there
exists $\tx \in \tS_{i-1}$ such that $\tx \sim \ty, \tw'$, by the triangle condition.

Note that $\tx\nsim \ta,\tb,\tc$ and that $x\sim w,w'$. By
Lemma~\ref{lem-aux}, $x \nsim a,b,c$. By the local triangle condition,
there exists a vertex $y$ adjacent to $x,b$, and $c$.  By (R$_i$)
there exists $\ty\in \tB_i$ such that $f_i(\ty)=y$ and $\ty\sim \tx$.
If $y\sim w$, first (R$_i$) implies that $\ty\sim \tw$ and then Lemma
\ref{lem-aux} shows that $\ty\sim\tb$.  If $y\nsim w$, $xwby$ is a
square of $G$ and by Lemma~\ref{lem-carre-tordu}, $\ty \sim \tb$ and
thus $\ty \in \tS_i$. Using the same reasoning, one can show that $\ty \sim
\tc$. Applying Case 1 to the triplet $\ty,\tb,\tc$, we conclude that $\tb\sim \tc$.
\end{proof}

We can now prove that the image under $f_{i+1}$ of a triangle
or a square is a triangle or a square.

\begin{lemma} \label{cellular}
If $\widetilde{a}\widetilde{b}\widetilde{c}$ is a triangle in $\tG_{i+1}$, then
$abc$ is a triangle in $G$. If
$\widetilde{a}\widetilde{b}\widetilde{c}\widetilde{d}$ is a square in
$\tG_{i+1}$, then $abcd$ is a square in $G$. Moreover, $\tG_{i+1}$ does not contain
induced $K_{2,3}$ and $W_4^-$. \end{lemma}

\begin{proof}
For triangles, the assertion follows directly from Lemma~\ref{lem-homomorphism}.
Consider now a square $\ta\tb\tc\td$. From
Lemmas~\ref{lem-homomorphism} and \ref{lem-locally-injective}, the vertices $a,
b, c,$ and $d$ are pairwise distinct and $a\sim b$, $b \sim c$, $c \sim d$, $d
\sim a$. From Lemma~\ref{lem-triangles}, $a\nsim c$ and $b \nsim
d$. Consequently, $abcd$ is a square in $G$.

Now, if $\tG_{i+1}$ contains an induced $K_{2,3}$ or $W^-_4,$ from the first
assertion
and Lemma \ref{lem-triangles} we conclude that the image under $f_{i+1}$ of this
subgraph will be an induced $K_{2,3}$ or  $W^-_4$
in the graph $G$, contrary to the interval condition.
\end{proof}

Lemma~\ref{cellular} implies that replacing all $3$--cycles and all
induced $4$--cycles of $\widetilde{G}_{i+1}$ by triangle- and
square-cells, we will obtain a triangle-square flag complex, which we
denote by $\widetilde{{\bX}}_{i+1}$.  Then, obviously, $\widetilde{G}_{i+1}=G(\widetilde{{\bX}}_{i+1})$.  The first assertion
of Lemma~\ref{cellular} and the flagness of $X$ implies that $f_{i+1}$
can be extended to a cellular map from $\widetilde{\bX}_{i+1}$ to
$\bX$: $f_{i+1}$ maps a triangle
$\widetilde{a}\widetilde{b}\widetilde{c}$ to the triangle $abc$ of
${\bX}$ and a square
$\widetilde{a}\widetilde{b}\widetilde{c}\widetilde{d}$ to the square
$abcd$ of ${\bX}$.

\begin{lemma}[Properties (R$_{i+1}$) and (T$_{i+1}$)]
\label{Ri+1}\label{Ti+1}
The map $f_{i+1}$ satisfies the conditions $(R_{i+1})$ and $(T_{i+1})$.
\end{lemma}

\begin{proof}
From Lemmas~\ref{lem-locally-injective} and~\ref{lem-triangles}, we
know that for any $\tw \in \tB_{i+1}$, $f_{i+1}$ induces an
isomorphism between the subgraph of $\tG_{i+1}$ induced by
$B_1(\tw,\tG_{i+1})$ and the subgraph of $G$ induced by
$f_{i+1}(B_1(\tw,\tG_{i+1}))$. Consequently, the condition (T$_{i+1}$)
holds.  From Lemma~\ref{lem-locally-surjective}, we know that for
every $\tw \in \tB_i$, $f_{i+1}(B_1(\tw,\tG_{i+1})) = B_1(w,G)$ and
consequently (R$_{i+1}$) holds as well.
\end{proof}

\begin{lemma}[Property (S$_{i+1}$)]
\label{Si+1} For any $\widetilde{w},\widetilde{w}'\in
  \widetilde{B}_{i}$ such that the vertices
  $w=f_{i+1}(\widetilde{w}),w'=f_{i+1}(\widetilde{w}')$ belong to a square
  $ww'u'u$ of ${\bX}$, there exist $\widetilde{u},\widetilde{u}'\in
  \widetilde{B}_{i+1}$ such that $f_{i+1}(\widetilde{u})=u,
  f_{i+1}(\widetilde{u}')=u',$ and
$\widetilde{w}\widetilde{w}'\widetilde{u}'\widetilde{u}$
  is a square of $\widetilde{{\bX}}_{i+1}$, i.e., $\widetilde{{X}}_{i+1}$ satisfies the property $(S_{i+1})$.
\end{lemma}

\begin{proof}
Note that if $\tw, \tw' \in \tB_{i-1}$, the lemma holds by the condition
(S$_i$). Let us assume further that $\tw \in \tS_i$.  By the property (R$_{i+1}$) (cf.\ Lemma~\ref{Ri+1}) applied to $\tw$ and $\tw'$, we know that in
$\tG_{i+1}$ there exist $\tu,\tu'$ distinct from $\tw,\tw'$, such that $\tu\sim\tw$,
$\tu'\sim\tw'$ and $f_{i+1}(\tu)=u$, $f_{i+1}(\tu')=u'$. Observe that, by (R$_{i+1}$), we have $\tu \nsim \tw'$ and $\tu' \nsim \tw$.

\begin{claim}\label{claim-noW4}
If there exists $y \notin \{u',w\}$ such that $y \sim u,w'$, then
$\tw\tw'\tu'\tu$ is a square in $\tG_{i+1}$.
If there exists $\ty \in \tB_{i+1}$ such that $\ty
\notin\{\tu',\tw\}$ and $\ty \sim \tu,\tw'$, then $\tw\tw'\tu'\tu$ is a square in $\tG_{i+1}$.
\end{claim}

\begin{proof} For the first statement: By the interval condition
  applied to the interval $I(u,w')$, we have $y \sim
  u,u',w,w'$. By (R$_{i+1}$) applied to $\tw$ and $\tw'$, there exists
  $\ty \in \tB_{i+1}$ such that $\ty \sim \tu,\tu',\tw$ and $f_{i+1}(\ty)=y$. By
  (R$_{i+1}$) or (T$_{i+1}$) applied to $\ty$, we have $\tu \sim \tu'$.

The second statement follows from the first one, and from the fact that, by (R$_{i+1}$), $y = f_{i+1}(\ty) \notin \{u',w\}$ and $y \sim u,w'$.
\end{proof}

Thus, for the rest of the proof of the lemma, we assume the following (since
otherwise the lemma follows from Claim~\ref{claim-noW4}):

\begin{itemize}
\item there does not exist
$\ty \in \tB_{i+1}$ such that
$\ty \sim
\tu,\tw'$ and $\ty \neq \tu',\tw$, or such that $\ty \sim \tu',\tw$ and $\ty \neq \tu,\tw'$;
\item there does not exist $y \in V(G)$ such that $y \sim u,w'$ and $y\neq u',w$,
or such that $y \sim u',w$ and $y\neq u,w'$.
\end{itemize}

\medskip
\noindent{\textbf{Case 1.} $\tw \in \tS_{i}, \tw' \in \tS_{i-1}$.}
\medskip

If $\tu' \in \tB_{i-1}$ then, by (S$_i$) applied to $\tw'$ and $\tu'$, we
conclude that $\tw\tw'\tu'\tu$ is a square in $\tG_{i+1}$. Hence further we assume
that $\tu'\in \tS_i$.

If $\tu \in \tS_{i-1}$ then, by (R$_{i+1}$) applied to $\tw$, we conclude that $\tu$ is not adjacent to $\tw'$.
By the square-pyramid condition (Q$_i$), there exists
 $\ty \in \tB_{i-1}$ such that $\ty \sim \tu,\tw'$, which contradicts our assumptions.

Suppose now that $\tu \in \tS_{i}$.  By the triangle condition
(Q$_i$), there exists $\ty\in \tS_{i-1}$ such that $\ty \sim
\tu,\tw$. By (R$_{i+1}$), we know that $y = f_i(\ty)
\notin\{u,u',w,w'\}$ and $y \sim u,w$.
By our assumptions, we have $y \nsim u',w'$. By the local triangle condition, there exists $x
\sim u',w',y$.  By (R$_{i}$) applied to $\tw'$ there exists $\tx \in
\tB_i$ such that $\tx \sim \tw', \tu'$ and $f_{i+1}(\tx)=x$.  Again, by our assumptions,
we have $x \nsim u,w$, i.e.,
$wyxw'$ is a square of $G$.  By the previous case applied to the
square $wyxw'$ (i.e., with $\ty$ and $\tx$ playing respectively the
roles of $\tu$ and $\tu'$), we get that $\tw\ty\tx\tw'$ is a square of
$\tG_{i+1}$. By the positioning condition (U$_i$) with respect to
$\tv$ for the square $\tw\ty\tx\tw'$, we get that $\tx \in
\tS_{i-2}$. This is impossible since $\tx \sim \tu'$ and $\tu' \in
\tS_i$.

Suppose now that $\tu \in \tS_{i+1}$, i.e., $\tu = [\tw,u]$. By
Lemma~\ref{lem-carre-tordu}, $(\tu',u) \in Z$. Since $\tw' \in
\tS_{i-1}$, by $(Z2)$, $\tu = [\tu',u] \sim \tu'$ and we are done.

\medskip
\noindent{\textbf{Case 2.} $\tw,\tw' \in \tS_{i}$.}
\medskip

By the triangle condition (Q$_i$), there exists $\ty \in \tS_{i-1}$
such that $\ty \sim \tw, \tw'$. By (R$_{i+1}$), we have $y =
f_{i+1}(\ty) \sim w,w'$.  By our assumptions we have
that $y \nsim u,u'$. By the local triangle condition, there exists $x
\sim y,u,u'$ and, again by our assumptions, we have $x\nsim w,w'$.
By (R$_{i+1}$), there exists $\tx \sim \ty$ such that
$f_{i+1}(\tx) =x$. Applying Case 1 to the squares $wyxu$ and $w'yxu'$, we get that
$\tx \sim \tu$ and $\tx \sim
\tu'$. By (T$_{i+1}$) applied to $\tx$, we conclude that $\tu \sim \tu'$.
\end{proof}

\begin{lemma}[Property (U$_{i+1}$)]
  The graph $\widetilde{{G}}_{i+1}$ satisfies the property (U$_{i+1}$), i.e.,
  the squares of $\tG_{i+1}$ satisfy the positioning condition PC$(\tv)$.
\end{lemma}

\begin{proof}
 Suppose by way of contradiction that there exists a square $\ta\tb\tc\td$ of $\tG_{i+1}$ such
 that
 \begin{align*}
 d(\ta,\tv)+d(\tc,\tv) < d(\tb,\tv)+d(\td,\tv). \tag{$\ast$}
 \end{align*}
 Let $a,b,c,d$ be the respective images of $\ta,\tb,\tc,\td$ in $G$ by $f_{i+1}$.
 By Lemma \ref{Ri+1}, vertices $a,b,c,d$ induce a square in $G$.
 If $\ta,\tb,\tc,\td \in \tB_{i}$, then (U$_i$) leads to a
 contradiction.  In view of ($\ast$), in the following we can assume that at least one of the vertices $\tb,\td$ belongs to
 $\tS_{i+1},$ say $\td \in
 \tS_{i+1}$. Consequently, $\ta,\tc \in \tS_i \cup \tS_{i+1}$ and $\tb
 \in \tS_{i-1} \cup \tS_i \cup \tS_{i+1}$. Moreover, by ($\ast$), we may assume
 without loss of generality that $\ta \in \tS_i$.

 \medskip
 \noindent{\textbf{Case 1.} $\tc \in \tS_i$.}
 \medskip

 Note that the inequality ($\ast$) implies that $\tb
 \in \tS_i \cup \tS_{i+1}$.  Since $\td = [\ta,d] = [\tc,d]$, either
 there exists $\tx \in \tS_{i-1}$ such that $\tx \sim \ta,\tc$, or
 there exist $\ty,\ty' \in \tS_{i}$ such that $\ty, \ty' \sim
 \ta,\tc,\td$ and $\ty\nsim\ty'$.

 If there exists $\tx \in \tS_{i-1}$ such that $\tx \sim \ta,\tc$,
 then $\tx \neq \tb$ since $\tb \in \tS_{i} \cup \tS_{i+1}$. By
 (R$_i$), we have $x = f_{i+1}(\tx) \sim a,c$ and $x \notin
 \{b,d\}$. By the interval condition applied to $I(a,c)$, we get that
 $x \sim b,d$. Consequently, by (R$_{i+1}$) (cf.\ Lemma~\ref{Ri+1}),
 we get $\tx \sim \td$. However, this is impossible since $\tx \in
 \tS_{i-1}$ and $\td \in \tS_{i+1}$.

 Suppose now that there exist $\ty,\ty' \in \tS_{i}$ such that $\ty,
 \ty' \sim \ta,\tc,\td$ and $\ty\nsim\ty'$. By (R$_{i+1}$), we have
 $y=f_{i+1}(\ty) \sim a,c,d$ and $y'= f_{i+1}(\ty') \sim a,c,d$. By the
 interval condition, $b\sim y,y'$, and by (R$_{i+1}$), we get $\tb\sim
 \ty, \ty'$. By the triangle condition, there exists $\tu \in
 \tS_{i-1}$ such that $\tu \sim \ta,\ty$.  Since $\tb \in \tS_{i} \cup \tS_{i+1}$, we have
 $\tu \neq \tb$, and by (R$_{i+1}$) we get $u = f_{i+1}(\tu) \neq b$,
 $u\sim a,y$, and $u \nsim d$. If $u \sim c$, by the interval
 condition applied to  $I(a,c)$, we have $u \sim d$. This is a
 contradiction and, therefore, $u \nsim c$.  Consider the triangles, $ayb, ayd,$ and
 $ayu$, all three sharing the common edge $ay$. By the no-propeller property (cf.\ Lemma~\ref{l:noprop}),
 we get that $b \sim u$ and by (R$_{i+1}$), we have $\tb\sim\tu$. Consequently, $\tb \in \tS_i$.

 By the triangle condition (Q$_i$), there exists $\tu' \in \tS_{i-1}$
 such that $\tu' \sim \tb,\tc$. By (R$_{i+1}$), we have $u' =
 f_{i+1}(\tu') \sim b,c$ and $u' \nsim d$. If $u' \sim a$, by the
 interval condition applied to  $I(a,c)$, we get $u' \sim d$, a contradiction.
 Hence, $u' \nsim a$. By Lemma~\ref{c:xexists}, we can assume
 that there exists $\tx \in \tS_{i-2}$ such that $\tx \sim
 \tu,\tu'$. By (R$_{i+1}$), we have $x = f_{i+1}(\tx) \sim u,u'$, and $x
 \nsim a,b,c$, i.e., $d(x,a) = d(x,b) = d(x,c) = 2$. By the local
 positioning condition applied to the square $abcd$ with respect to $x$, we get
 that $d(x,d) = 2$. By the local triangle condition, there exists $z
 \sim a,d,x$. By (R$_{i+1}$), there exists $\tz \sim \tx$ such
 that $f_{i+1}(\tz)=z$. If $z \sim u$, by (R$_{i+1}$) applied to
 $\tx$ and $\tu$, we get that $\tz \sim \tu,\ta$. If $z \nsim u$,
 by (S$_i$) applied to the square $xuaz$, we also get that
 $\tz \sim \ta$. By (R$_{i+1}$) applied to $\ta$, we get that
 $\tz \sim \td$. However, this is impossible since $\tz$ cannot be
 adjacent to $\td \in \tS_{i+1}$ and $\tx \in \tS_{i-2}$.

 \medskip
 \noindent{\textbf{Case 2.} $\tc \in \tS_{i+1}$.}
 \medskip

Since  $d(\ta,\tv)+d(\tc,\tv) < d(\tb,\tv)+d(\td,\tv)$ by ($\ast$),
we get $\tb \in \tS_{i+1}$. By (Q$_{i+1}$) (cf.\ Lemma~\ref{Qi+1}), there exists $\tu \in \tS_i$ such that $\tu \sim
\tb,\tc$. By (T$_{i+1}$) (cf.\ Lemma~\ref{Ri+1}), we have $u = f_{i+1}(\tu) \sim b,c$ and $u \sim
a$ iff $\tu \sim \ta$ (respectively, $u \sim d$ iff $\tu \sim \td$). By the
interval condition, $u \sim a$ iff $u \sim d$.

Suppose that $\tu \sim \ta$ (note that then $\tu \sim \td$). Then, by the triangle condition, there
exists $\tx \in \tS_{i-1}$ such that $\tx \sim \tu,\ta$. By (R$_{i+1}$), we have $x = f_{i+1}(\tx) \sim a,u$ and $x \nsim b,c,d$.
Consider the triangles, $aux, aub$ and $aud$, all three sharing the common edge $au$. By the no-propeller property (Lemma~\ref{l:noprop}), we get a contradiction
since $x,b,d$ are pairwise non-adjacent.

Thus $\tu \nsim \ta$. By the square-pyramid condition (Q$_{i+1}$) (cf.\ Lemma~\ref{Qi+1}), we obtain two cases (i) and (ii) below:
\medskip

(i) There exists $\tx \in \tS_{i-1}$ with $\tx \sim \tu,\ta$. Then  the vertices $\tx,\ta,\td,\tu,\tb,\tc$ induce a half open book. By properties (R$_{i+1}$) and (T$_{i+1}$) (cf.\ Lemma~\ref{Ri+1}), we obtain that their images $x,a,d,u,b,c$ induce a half open book in $G$, which contradicts Lemma~\ref{l:noprop}.

\medskip

(ii) There exist two non-adjacent vertices $\ty,\ty' \in \tS_i$, both adjacent to $\tu,\ta,\tb$. By (R$_{i+1}$), we get then three triangles $uby,uby'$, and $ubc$ sharing the common edge $ub$. By the no-propeller property (Lemma~\ref{l:noprop}), we conclude that $y\sim c$ or $y'\sim c$, say $y\sim c$. By (R$_{i+1}$), we get
$\ty\sim \tc$, which reduces the case to the (impossible) situation when $\tu \sim \ta$ (obtained by replacing $\tu$ with $\ty$).
\medskip

In all the cases we assumed the inequality ($\ast$) and reached a contradiction. This implies that $d(\ta,\tv)+d(\tc,\tv) = d(\tb,\tv)+d(\td,\tv)$
and establishes the positioning condition.
\end{proof}

\subsection{The universal cover $\tX$.}
\label{s:univ}
Concluding our inductive construction, in this subsection we define
the universal covering map $\tX \to X$ and finish the proof of Theorem
\ref{t:m2} by showing that $\tX$ satisfies the interval and the
positioning conditions.
\medskip

\noindent
Let $\widetilde{\bX}_v$ denote the triangle-square complex obtained as the
directed union
$\bigcup_{i\ge 0} \widetilde{\bX}_i$ with the vertex $v$ of $\bX$ as the
basepoint. Denote by $\tG_v$ the
$1$--skeleton of $\widetilde{\bX}_v$. Let
$f=\bigcup_{i\ge 0}f_i$ be the cellular map from $\widetilde{\bX}_v$ to $\bX$.

\begin{lemma} \label{covering_map}
For any $\tw \in \widetilde{\bX}$, the restriction
$f|_{\mbox{St}(\tw,{\widetilde{\bX}_v})}$ of $f$ is an isomorphism
between the stars $\mbox{St}(\tw,{\widetilde{\bX}_v})$ and
$\mbox{St}(w,\bX)$.
Consequently, the map $f\colon\widetilde{{\bX}}_v\rightarrow~\!\!\bX$
is a covering map.
\end{lemma}

\begin{proof}
Note that, since $\widetilde{\bX}_v$ is a flag complex, a vertex $\tx$
of $\widetilde{\bX}_v$ belongs to $\mbox{St}(\tw,\widetilde{\bX}_v)$ if and only if either $\tx \in
B_1(\tw,\tG_v)$ or $\tx$ has two non-adjacent neighbors in
$B_1(\tw,\tG_v)$.

Let $\tw\in \tS_i$, i.e., $i$ is the distance between $\tv$ and $\tw$ in
$\tG_v$,
and consider the set $B_{i+2}(\tv,\tG_v)$. Then the vertex-set of
$\mbox{St}(\tw,\widetilde{\bX}_v)$ is included in $B_{i+2}(\tv,\tG_v)$. From
(R$_{i+2}$) we know that $f$ is an isomorphism
between the graphs induced by $B_1(\tw,\tG_v)$ and $B_1(w,G)$.

For any vertex $x$ in  $\mbox{St}(w,\bX)\setminus B_1(w,G)$ there exists a square $wuxu'$ in $G$. By (R$_{i+2}$), there exist $\tu,
\tu'$ both adjacent to $\tw$ in $\tG_v$ and such that $\tu \nsim \tu'$, and $f(\tu)=u,f(\tu')=u'$. By (S$_{i+2}$) applied
to $\tw,\tu$ and since $\tw$ has a unique neighbor $\tu'$ mapped to
$u'$, there exists a vertex $\tx$ in  $\tG_v$ such that $f(\tx) = x$, $\tx \sim
\tu, \tu'$ and $\tx \nsim \tw$. Consequently, $f|_{V\left(\mbox{St}(\tw,\widetilde{\bX}_v)\right)}$ is a surjection
from $V(\mbox{St}(\tw,\widetilde{\bX}_v))$ onto $V(\mbox{St}(w,\bX))$.

Now we show that $f|_{V\left(\mbox{St}(\tw,\widetilde{\bX}_v)\right)}$
is injective.
Suppose by way of contradiction that there exist two distinct vertices
$\tu, \tu'$ of $\mbox{St}(\tw,\widetilde{\bX}_v)$ such that $f(\tu) =
f(\tu') = u$. If $\tu,\tu' \sim \tw$, by condition (R$_{i+1}$) applied
to $\tw$, we get a contradiction.  Suppose first that $\tu \sim \tw$
and $\tu' \nsim \tw$ and let $\tz \sim \tw, \tu'$. This implies that
$w, u, z$ are pairwise adjacent in $G$. Since $f$ is an isomorphism
between the graphs induced by $B_1(\tw,\tG_v)$ and $B_1(w,G)$, we
conclude that $\tz \sim \tu$. But then $f$ is not locally injective
around $\tz$, contradicting the condition (R$_{i+2}$).  Suppose now
that $\tw \nsim \tu, \tu'$.  Let $\ta \nsim \tb$, respectively $\ta'
\nsim \tb'$, be vertices adjacent to both $\tu$ and $\tw$, and,
respectively, $\tu'$ and $\tw$.  If $\ta' = \ta$ or $\ta' = \tb$, then
applying (R$_{i+2}$) to $\ta'$, we get that $f(\tu)\neq
f(\tu')$. Hence further we suppose that $\ta' \notin \{\ta,\tb\}$.  By
(R$_{i+1}$) applied to $\tw$ we have that $a'\neq a\neq b \neq a'$ and
$a\nsim b$. In $G$, the vertices $a,b,a',b'$ belong to the interval  $I(w,u)$.
Consequently, by the interval condition, $a' \sim a,b$. By
(R$_{i+2}$) applied to $\tw$ and $\ta$, $\ta'\sim \ta$ and $\ta' \sim
\tu$. Thus, by (R$_{i+2}$) applied to $\ta'$, $\tu = \tu'$,
contradicting our choice of $\tu,\tu'$. In all cases, we get a
contradiction, thus $\tu$ and $\tu'$ as above do not exist.

Hence $f|_{V\left(\mbox{St}(\tw,\widetilde{\bX}_v)\right)}$ is a bijection between the vertex-sets of
$\mbox{St}(\tw,\widetilde{\bX}_v)$ and $\mbox{St}(w,\bX)$.
We show now that $\tu\sim \tu'$ in
$\mbox{St}(\tw,\widetilde{\bX}_v)$ if and only if $u \sim u'$ in $\mbox{St}(w,\bX)$.
If $\tu\sim \tu'$ then $u\sim u'$ by (R$_{i+2}$).
Assume now that $\tu \nsim \tu'$ and $u\sim u'$.
By (R$_{i+2}$), there exists $\tu'' \sim \tu$ with $f(\tu'')=u'$. This leads however to a contradiction
by the local injectivity of $f$.

By (R$_{i+2}$)
applied to $w$ and since $\bX$ and $\widetilde{\bX}_v$ are flag complexes,
$\ta\tb\tw$ is a triangle in $\mbox{St}(\tw,\widetilde{\bX}_v)$ if and only if $abw$ is
a
triangle in $\mbox{St}(w,\bX)$. By (R$_{i+2}$) and since $\bX$ is a flag
complex, if $\ta\tb\tc\tw$ is a square in $\mbox{St}(\tw,\widetilde{\bX})$, then $abcw$
is a square in $\mbox{St}(w,\bX)$. Conversely, by the conditions (R$_{i+2}$)
and (S$_{i+2}$) and the
flagness of $\widetilde{\bX}_v$, we conclude that if $abcw$ is a square in
$\mbox{St}(w,\bX)$,
then $\ta\tb\tc\tw$ is a square in $\mbox{St}(\tw,\widetilde{\bX}_v)$. Consequently, for
any $\tw \in \widetilde{\bX}_v$, the map $f|_{\mbox{St}(\tw,\widetilde{\bX}_v)}$ is an
isomorphism between $\mbox{St}(\tw,\widetilde{\bX}_v)$ and $\mbox{St}(w,\bX)$, and
thus $f$ is a
covering map.
\end{proof}

\begin{lemma} \label{house-condition-tX}
The graph $\widetilde{G}=G(\widetilde{\bX}_v)$ satisfies
the interval condition and the positioning condition with respect to $\tv$.
\end{lemma}
\begin{proof}
  For every $\tw,\tw' \in V(\tG)$ such that $d(\tw,\tw')=2$, by
  Lemma~\ref{covering_map}, $d(w,w') = 2$, and by the interval
  condition in $G$, there exists a square $uwu'w'$ in $\mbox{St}(w,\bX)$. By
  Lemma~\ref{covering_map}, $\tw' \in V(\mbox{St}(\tw,\tX))$, and the
  interval $I(\tw,\tw')$ is contained in $\mbox{St}(\tw,\tX)$.
  Since the map $f|_{\mbox{St}(\tw,\widetilde{\bX}_v)}$ is an
  isomorphism onto its image (cf.\ Lemma \ref{covering_map}), the interval condition
  for $I(\tw,\tw')$ is satisfied.

  The positioning condition with respect to $\tv$, i.e.\ PC($\tv$), is a consequence of (U$_i$)
  for sufficiently large $i$.
\end{proof}

\begin{lemma}\label{simplyconnected}
  The complex
  $\widetilde{\bX}_v$ is simply-connected for
  any basepoint $v\in V({\bX})$. For any two vertices $\tv$ and $\tv'$ the corresponding complexes $\tX_v$ and $\tX_{v'}$ are isomorphic.
\end{lemma}
\begin{proof}
  Simple connectedness follows from Lemma \ref{l:simconn} and from the fact that $\tX_v$ satisfies the condition (Q$_i$) for every $i$.
  It follows then, by Lemma \ref{covering_map}, that $\tX_v$ is the universal cover of $X$, and the second statement is a consequence of the uniqueness of the universal cover (cf.\ e.g.\ \cite[page 67]{Hat}).
\end{proof}

Thus, for any choice of the basepoint we obtain the same universal
cover $\tX$ of $X$.  By Lemma \ref{house-condition-tX}, its
$1$--skeleton satisfies the interval and the positioning conditions.
This finishes the proof of Theorem \ref{t:m2}.

\section{Examples and extensions}
\label{s:rem}\label{s:DeltaM}

\subsection{Examples}
 \label{s:eg}
 Here, we provide examples of graphs satisfying our local conditions and not being the basis graphs of matroids. Of course, in view of Theorem \ref{t:m2} such examples arise as quotients of basis graphs of matroids under free actions of groups --- for basics on relations between group actions and covering spaces see e.g.\ \cite[Chapter 1.3]{Hat}. Note, that the quotient should be a graph (without multiple loops etc.) so that the displacement function for the group action should be large enough. For example, there is no such nontrivial action on $C_4$.

In fact our examples are the same as examples given in \cite[Theorem 2.3]{DHT} for slightly different purposes (see comments below). We follow the notations of \cite{DHT}.
Let $B_{n,n}$ be the basis graph of the complete matroid $M_{n,n}$, i.e., the one formed by the family of all $n$--element subsets of a set of cardinality $2n$.
Define a $\mathbb Z_2$--action on $B_{n,n}$ in the way that each vertex $v$ of $B_{n,n}$ is mapped by the generator of $\mathbb Z_2$ to the antipodal vertex $v^{\ast}$, i.e., the unique vertex at distance $n$ from $v$ (this is in fact the vertex corresponding to the complement of the $n$--element set $v$). It is easy to observe that this defines an action by graph automorphisms and that, for $n\geq 2$, this action is free.
It can be observed, that a combinatorial ball of radius $\lfloor n/2 \rfloor-1$ in $H_n:=B_{n,n}/\mathbb Z_2$ is isomorphic to a ball of the same radius in $B_{n,n}$.
Thus, for $n\geq 8$, balls of radii up to $3$ look as corresponding balls in $B_{n,n}$, i.e., $H_n$ satisfies our local conditions. Moreover, for such $n$, the quotient map $B_{n,n}\to H_n$ induces a map of the corresponding triangle-square complexes $X(B_{n,n})\to X(H_n)$, being a covering map. It follows that $\pi_1(X(H_n))=\mathbb Z_2$, and hence $H_n$ is not the basis graph of a matroid.

\begin{remark}
 It is stated in \cite{DHT} (cf.\ discussion after Theorem 2.6 there) that ``the graphs $H_n$ offer counterexamples to any number of futile conjectures(...), including Conjectures 2 and 3 of Maurer's thesis \cite{Mau}". As shown by our result a general form of Maurer's Conjecture 3 --- saying that the triangle-square complexes of basis graphs of matroids may be characterized as simply connected complexes satisfying some local conditions --- is true. In fact, as shown above, the existence of graphs $H_n$ is consistent with the picture, since the corresponding complexes are not simply connected for large $n$.
\end{remark}

\begin{remark}
 Note that the counterexamples to the original Maurer's Conjecture 3 \cite{Mau} provided in
\cite{DHT} do not satisfy our local conditions. The second example, cf.\ \cite[Fig.\ 1]{DHT}, does
not satisfy the local positioning condition, while
the first example, cf.\ \cite[Fig.\ 3]{DHT}, does not even satisfy the local triangle condition.
\end{remark}

\subsection{Extension to even $\Delta$--matroids}
\label{s:ext}
Now, we will show that our Theorem \ref{main_th} can be extended to
even $\Delta$--matroids.  A \emph{$\Delta$--matroid} is a collection
$\mathcal B$ of subsets of a finite set $I$, called \emph{bases} (not
necessarily equicardinal) satisfying the symmetric exchange property:
for any $A,B\in {\mathcal B}$ and $a\in A\Delta B$, there exists $b\in
A\Delta B$ such that $A\Delta\{ a,b\}\in {\mathcal B}$. A $\Delta$--matroid whose bases all have the same cardinality modulo $2$ is
called an \emph{even $\Delta$--matroid}. The \emph{basis graph}
$G=G({\mathcal B})$ of an even $\Delta$--matroid $\mathcal B$ is the
graph whose vertices are the bases of $\mathcal B$ and edges are the
pairs $A,B$ of bases differing by a single exchange, i.e., $\vert
A\Delta B\vert=2$. Extending Maurer's characterization of basis
graphs of matroids, it was shown in \cite{Che_bas} that a graph $G$ is the basis graph of an even $\Delta$--matroid if and only if $G$ satisfies
the positioning condition, the \emph{generalized link condition} (the
neighborhood of each vertex is the line graph of a finite graph) and the
\emph{generalized interval condition} (IC4) (each $2$--interval of $G$
contains a square and is an induced subgraph of the $4$--dimensional
octahedron). It was also noted in \cite{Che_bas} that the generalized
link condition is necessary, i.e., the interval condition (IC4) and
the positioning condition solely do not characterize basis graphs
of even $\Delta$--matroids. Wenzel \cite{Wen} showed that the triangle-square complexes defined by basis graphs of even $\Delta$-matroids are simply
connected.

Let $G$ be a (not necessarily finite) graph satisfying the local
positioning, the generalized link and the generalized interval
conditions. Inspecting the proofs of Lemma~\ref{l:noprop} and Theorem
\ref{t:m2} (namely, noting that each use of the interval condition
either employs a square or a pyramid in a $2$--interval), analogously one
can conclude that the $1$--skeleton of the universal cover
$\widetilde{X}=\widetilde{X(G)}$ of the triangle-square complex $X(G)$
of $G$ satisfies the positioning condition and the generalized
interval condition (IC4). Now, for any choice of the basepoint $v$,
the triangle-square complex $\widetilde{X}_v$ is isomorphic to
$\widetilde{X}$. Since the neighborhood of $\tilde{v}$ in the
$1$--skeleton of $\widetilde{X}_v$ coincides with $N(v)$ and thus is a
line graph by the generalized link condition, we conclude that the
$1$--skeleton $G(\widetilde{X})$ of $\widetilde{X}$ satisfies the
generalized link condition. From the result of \cite{Che_bas} it
follows that $G(\widetilde{X})$ is the basis graph of an even
$\Delta$--matroid, thus establishing the following result:

\begin{theorem}\label{th_Delta}
  For a graph $G$ the following conditions are equivalent:
  \begin{enumerate}[(i)]
    \item
    $G$ is the basis graph of an even $\Delta$--matroid;
    \item
    the triangle-square complex $X(G)$ is simply connected, every ball of radius $3$ in $G$ is isomorphic to a ball of radius $3$ in
    the basis graph of an even $\Delta$--matroid;
    \item
    the triangle-square complex $X(G)$ is simply connected, $G$
    satisfies the generalized interval condition (IC4), the generalized link condition, and the local positioning condition.
  \end{enumerate}
\end{theorem}

%============ B I B L I O G R A P H Y ==============================

\medskip

\noindent
{\bf Acknowledgements.} J.C. was partially supported by ANR project MACARON (\textsc{anr-13-js02-0002}).  V.C.\ was partially supported ANR projects TEOMATRO (\textsc{anr-10-blan-0207}) and GGAA (\textsc{anr-10-blan-0116}). D.O.\ was partially supported by the MNiSW grant N201 541738, and by Narodowe Centrum Nauki, decision no.\ DEC-2012/06/A/ST1/00259. The largest part of this
work was carried out while D.O.\ was visiting Aix-Marseille Universit\'e in September 2012.
The visit was partially supported by the ERC grant ANALYTIC no.\ 259527.
The research conditions and the environment  provided by the host institution are greatly
appreciated.

\end{document}